\documentclass[final]{siamltex}
\usepackage{latexsym,amssymb}
\usepackage{epsfig}
\newcommand{\E}{{\bf E}}
\newcommand{\Hb}{{\bf H}}
\newcommand{\me}{\mathcal{E}}
\newcommand{\mh}{\mathcal{H}}
\newcommand{\mE}{\mathcal{E}}
\newcommand{\mH}{\mathcal{H}}
\newcommand{\ds}{\displaystyle}
\newcommand{\no}{\nonumber}
\newcommand{\be}{\begin{eqnarray}}
\newcommand{\ben}{\begin{eqnarray*}}
\newcommand{\en}{\end{eqnarray}}
\newcommand{\enn}{\end{eqnarray*}}

\newcommand{\pa}{\partial}

\newcommand{\g}{\gamma}

\newcommand{\vep}{\varepsilon}
\newcommand{\Om}{\Omega}

\newcommand{\sig}{\sigma}

\newcommand{\al}{\alpha}
\newcommand{\la}{\lambda}

\newcommand{\dt}{\Delta t}
\newcommand{\dx}{\Delta x}
\newcommand{\dy}{\Delta y}
\newcommand{\dz}{\Delta z}
\newtheorem{remark}[theorem]{Remark}

\title{Optimal error estimates and energy conservation identities of the ADI-FDTD
scheme on staggered grids for 3D Maxwell's equations}
\author{Liping Gao\thanks{School of Mathematics and Computational Sciences,
China University of Petroleum, 66 Changjiang West Road, Qingdao 266555,
Shandong Province, China ({\tt  l.gao@upc.edu.cn})}
\and
Bo Zhang\thanks{LSEC and Institute
of Applied Mathematics, AMSS, Chinese Academy of Sciences, Beijing
100190, China ({\tt b.zhang@amt.ac.cn})}}

\begin{document}
\maketitle

\begin{abstract}
This paper is concerned with the optimal error estimates and energy conservation properties
of the alternating direction implicit finite-difference time-domain (ADI-FDTD) method which is a
popular scheme for solving the 3D Maxwell equations.
Precisely, for the case with a perfectly electric conducting (PEC) boundary condition
we establish the optimal second-order error estimates in both space and time in the
discrete $H^1$-norm for the ADI-FDTD scheme and prove the approximate divergence preserving
property that if the divergence of the initial electric and magnetic fields are zero then
the discrete $L^2$-norm of the discrete divergence of the ADI-FDTD solution is
approximately zero with the second-order accuracy in both space and time.
A key ingredient is two new discrete energy norms which are second-order in time perturbations
of two new energy conservation laws for the Maxwell equations introduced in this paper.
Furthermore, we prove that, in addition to two known discrete energy identities which are
second-order in time perturbations of two known energy conservation laws,
the ADI-FDTD scheme also satisfies two new discrete energy identities which are
second-order in time perturbations of the two new energy conservation laws.
This means that the ADI-FDTD scheme is unconditionally stable under the four
discrete energy norms.
%Note that it is a superconvergence result that the ADI-FDTD scheme is second-order convergent in
%space under the discrete $H^1$ semi-norm.
Experimental results are presented which confirm the theoretical results.
\end{abstract}



\pagestyle{myheadings}
\thispagestyle{plain}
\markboth{L. GAO, B. ZHANG }{OPTIMAL ERROR ESTIMATE AND ENERGY CONSERVATION FOR ADI-FDTD}

\section{Introduction}

The alternating direction implicit finite-difference time-domain (ADI-FDTD) method
was first proposed in 2000 in \cite{Namiki,Zheng1} for the three-dimensional Maxwell equations
with an isotropic medium. This method is based on the staggered grids of Yee's scheme \cite{Yee}
and splitting of Maxwell's equations, consists of only two stages for each time step
and is unconditionally stable.
The ADI-FDTD scheme is a popular scheme for solving the three-dimensional Maxwell equations
and is applicable to a wide range of problems in computational electromagnetics (see \cite{Taflove}).
So far, the ADI-FDTD scheme has been intensively studied by many authors
(see, e.g. \cite{DSSW,Fornberg,FT,Garcia,Garcia1,Gao1,Gedney,OPD,ST1,ST2,T,Zheng1,Zhao}),
such as the stability and dispersion analysis using Fourier analysis \cite{FT,Garcia1,OPD,ST1,Zheng1,Zhao},
the unconditional stability and convergence analysis by the energy method \cite{Fornberg,Gao1},
its accuracy and efficient algorithms \cite{DSSW,Garcia,ST2,T} and its applicability with
perfectly matched layer absorbing boundary conditions \cite{Gedney}.
Meanwhile, motivated by the success of the ADI-FDTD scheme, there is also considerable interest in
developing new unconditionally stable schemes for solving the Maxwell equations,
based on splitting of the Maxwell equations (see, e.g.
\cite{chen,chen1,Fornberg,Gao2,Gao3,Gao4,Kong,Lee,SCC}), where the stability and convergence of
these schemes were studied in the discrete $L^2$ norm.

However, optimal error estimates and energy conservation and divergence preserving properties
of the ADI-FDTD scheme have not yet been studied although these properties are vital for the
successful application in practical problems of the ADI-FDTD scheme.
The purpose of the present paper is to conduct such a study.
In particular, for the case of perfectly conducting (PEC) boundary conditions
we will establish optimal error estimates in the discrete $H^1$-norm for the
ADI-FDTD scheme and prove that the ADI-FDTD scheme satisfies four discrete energy identities
which are second-order in time perturbations of four energy conservation laws,
employing the energy method.
We also prove an approximate divergence preserving property with a second-order accuracy
in both space and time for the ADI-FDTD scheme.
In doing so, a crucial tool is two new discrete energy norms (cf. Theorems \ref{t3.1} and \ref{t3.2}
in Subsection \ref{sec3.3}) which are second-order in time perturbations of
two new energy conservation laws for the Maxwell equations (see Theorems \ref{t2.1}
and \ref{t2.2} below) introduced in the present paper.

The remaining part of the paper is organized as follows. In Section \ref{sec2},
we derive two new energy conservation laws for the Maxwell equations with the PEC boundary
conditions and introduce two well-known energy conservation laws.
In Section \ref{dis-energy} we recall the ADI-FDTD scheme, define some new energy norms and
prove that the ADI-FDTD scheme satisfies four energy identities which are second-order in time
perturbations of the four energy conservation laws for the Maxwell equations
introduced in Section \ref{sec2}.
This means that the ADI-FDTD scheme is unconditionally stable under the four discrete energy norms.
By making use of the two new discrete energy norms
we derive optimal second-order error estimates in both space
and time under the discrete $H^1$ norm for the ADI-FDTD scheme
(see Theorems \ref{t4.1}, \ref{t4.2} and \ref{t4.3}) in Section \ref{s-con} and
the approximate divergence preserving property (see Theorem \ref{t5.1}) in Section \ref{div-p}.
Note that it is a superconvergence result that the ADI-FDTD scheme is second-order convergent in
space under the discrete $H^1$ norm.
Experimental results are presented in Section \ref{exp} which confirm the theoretical results.

%For simplicity, our discussion will be restricted to the case of constant
%coefficients. However, extension to the case of variable coefficients is straightforward.

\section{Maxwell's equations and their conservation properties}\label{sec2}

In this section we introduce four conservation properties of
Maxwell's equations, of which two are new.
These conservation properties will be used to study
%the stability and convergence of
the ADI-FDTD scheme in this paper.

\subsection{Maxwell's equations}

Maxwell's equations governing the propagation of electromagnetic waves
can be written as a system of partial differential equations (see \cite{Taflove}):
\be\label{mw1}
\nabla\times\E =-\frac{\pa{\bf B}}{\pa t},%\\ \label{mw2}
\quad \nabla\times\Hb ={\bf J}+\frac{\pa {D}}{\pa t},%\\ \label{mw3}
\quad \nabla\cdot{\bf B}=0,%\\ \label{mw4}
\quad\nabla\cdot{\bf D}&=&\rho,
\en
where $\E$ is the electric field, $\Hb$ is the magnetic field,
${\bf B}=\mu\Hb$ is the magnetic flux density, ${\bf D}=\vep\E$ is
the electric displacement, ${\bf J}=\sigma\E$ is the current density,
$\vep$ is the electric permittivity, $\mu$ is the magnetic permeability,
$\sigma$ is the electric conductivity and $\rho$ is the volume
density of electric charge.
%(\ref{mw1}) is called Ampere's law and (\ref{mw2}) is Faraday's law.
%Starting from the two equations and the equation
%of continuity (or the equation of mass conservation): $\nabla\cdot
%{\bf J}=-\frac{\pa \rho}{\pa t}$, where  the following two equations
%(also called Gauss's laws) can be derived:

In this paper we consider Maxwell's equations in a lossless and
isotropic medium without a source field. In this case, $\mu$ and
$\vep$ are constant and $\rho=\sig=0$, so Maxwell's equations become
\be\label{mw2.1}
\frac{\pa E_x}{\pa t}=\frac{1}{\vep}\left(\frac{\pa H_z}{\pa y}
     -\frac{\pa H_y}{\pa z}\right), %\\ \label{mw2.2}
\quad\frac{\pa E_y}{\pa t}=\frac{1}{\vep}\left(\frac{\pa H_x}{\pa z}
     -\frac{\pa H_z}{\pa x}\right),\\ \label{mw2.3}
\frac{\pa E_z}{\pa t}=\frac{1}{\vep}\left(\frac{\pa H_y}{\pa x}
     -\frac{\pa H_x}{\pa y}\right), %\\ \label{mw2.4}
\quad\frac{\pa H_x}{\pa t}=\frac{1}{\mu}\left(\frac{\pa E_y}{\pa z}
     -\frac{\pa E_z}{\pa y}\right),\\ %\label{mw2.5}
\frac{\pa H_y}{\pa t}=\frac{1}{\mu}\left(\frac{\pa E_z}{\pa x}
     -\frac{\pa E_x}{\pa z} \right), %\\
\quad\frac{\pa H_z}{\pa t}=\frac{1}{\mu}\left(\frac{\pa E_x}{\pa y}
     -\frac{\pa E_y}{\pa x}\right),\label{mw2.6}
\en
where $\E=(E_x,E_y,E_z)$ and $\Hb=(H_x,H_y,H_z)$ denote the the electric and
magnetic fields, respectively.
For simplicity we consider the perfectly electric conducting (PEC) condition
on the boundary $\pa\Om$ of the cubic domain $\Omega=[0,1]\times[0,1]\times[0,1]$:
\be\label{pec2.1}
\vec{n}\times\E = 0,\;\; \;\;\vec{n}\cdot\Hb=0,\;\;\;\mbox{on}\;\;(0,T]\times\pa\Omega,
\en
where $T>0$ and $\vec{n}$ is the outward normal vector of $\pa\Om$.
We also need the initial conditions:
\be\label{ic}
\E(0,x,y,z)=\E_0(x,y,z)\qquad\mbox{and}\qquad \Hb(0,x,y,z)=\Hb_0(x,y,z).
\en
It is well known that, for suitably smooth data, the problem (\ref{mw2.1})-(\ref{ic})
has a unique solution for all time (see \cite{Leis}).

\subsection{Energy conservation laws in a lossless medium}\label{cont-energy}

When the medium is lossless and isotropic, we have the following two new laws
of energy conservation.

\begin{theorem}(Energy Conservation I)\label{t2.1}
Let $\E$ and $\Hb$ be the solution of the Maxwell equations $(\ref{mw2.1})-(\ref{mw2.6})$
in a lossless and isotropic medium satisfying the PEC boundary conditions $(\ref{pec2.1})$.
Then for $w=x,y,z$,
\be\label{thm2.1}
\int_\Om\left(\vep\left|\frac{\pa\E}{\pa w}\right|^2
 +\mu\left|\frac{\pa\Hb}{\pa w}\right|^2\right)dxdydz\equiv\mbox{Constant}.
\en
\end{theorem}

\begin{proof}
%Let $\Pa\Om_x=\{(0,y,z),(1,y,z)||y|\leq 1,|z|\leq 1\}$, the
%boundary parallel to the $yoz$ plane.
We only prove (\ref{thm2.1}) for the case with $w=x.$
The cases with $w=y,z$ can be shown similarly.

Differentiating the first two equations in (\ref{mw1})
with respect to $x$, we have
%\small
\be\label{thm2.11}
\nabla\times\frac{\pa\E}{\pa x}=-\frac{\pa^2(\mu\Hb)}{\pa x\pa t},\;\;\;\;
\nabla\times\frac{\pa\Hb}{\pa x}=\frac{\pa^2(\vep\E)}{\pa x\pa t}.
\en
%\normalsize
Taking the scalar product of the above two equations with ${\pa\Hb}/{\pa x}$ and
${\pa\E}/{\pa x}$, respectively, integrating the resulting equations over $\Om$ and
adding them together, we obtain on integrating by parts that
%\small
\be\label{thm2.12}
\int_\Om\left[\frac{\pa^2(\mu\Hb)}{\pa x\pa t}\cdot\frac{\pa\Hb}{\pa x}+
             \frac{\pa^2(\vep\E)}{\pa x\pa t}\cdot\frac{\pa\E}{\pa x}\right]dxdydz
=-\int_{\pa\Om}(\vec{n}\times\frac{\pa\E}{\pa x})\cdot\frac{\pa\Hb}{\pa x}ds.
\en
%\normalsize
%Hereafter, $dv=dxdydz$ denotes the element of volume.

Now, from the PEC boundary condition (\ref{pec2.1}) we have
$\ds\vec{n}\times{\pa\E}/{\pa x}=0$ for $y=0,1$ or for $z=0,1.$

Then (\ref{thm2.12}) becomes
%\small
\be\label{thm2.13a}
\frac{1}{2}\frac{\pa}{\pa t}\left[\int_\Om\left(\mu\left|\frac{\pa\Hb}{\pa x}\right|^2
+\vep\left|\frac{\pa\E}{\pa x}\right|^2\right)dxdydz\right]=T_0+T_1,
\en
%\normalsize
where
%\small
\ben
T_i=\int_0^1\int_0^1\left[\frac{\pa E_y}{\pa x}\frac{\pa H_z}{\pa x}
-\frac{\pa E_z}{\pa x}\frac{\pa H_y}{\pa x}\right]_{x=i}dydz,\;\;i=0,1.
%T_2&=&\int_0^1\int_0^1\left[\frac{\pa E_z}{\pa x}\frac{\pa H_y}{\pa x}
%-\frac{\pa E_y}{\pa x}\frac{\pa H_z}{\pa x}\right]_{x=1}dydz.
\enn
%\normalsize
Using the Maxwell's equations and the PEC boundary condition it can be shown that $T_0=T_1=0.$
Using the second equation in (\ref{mw2.1}) and the first equation in (\ref{mw2.3}) we have
\ben
T_0&=&\int_0^1\int_0^1\left[\left(\frac{\pa H_x}{\pa z}\frac{\pa E_y}{\pa x}
     -\frac{\pa H_x}{\pa y}\frac{\pa E_z}{\pa x}\right)
     -\vep\left(\frac{\pa E_y}{\pa t}\frac{\pa E_y}{\pa x}
     +\frac{\pa E_z}{\pa t}\frac{\pa E_z}{\pa x}\right)\right]_{x=0}dydz\\
   &=&\int_0^1\int_0^1\left[\frac{\pa H_x}{\pa z}\frac{\pa E_y}{\pa x}
     -\frac{\pa H_x}{\pa y}\frac{\pa E_z}{\pa x}\right]_{x=0}dydz
\enn
since, by the PEC boundary condition $\vec{n}\times\E=0$ on $\pa\Om$ in (\ref{pec2.1}),
we have $\pa E_y/\pa t=\pa E_z/\pa t=0$ for $x=0$ and for all $t>0$, $y\in\Om.$
By integration by parts it then follows that
\ben
T_0&=&\int_0^1\left[H_x\frac{\pa E_y}{\pa x}\Big|_{x=0,z=1}
      -H_x\frac{\pa E_y}{\pa x}\Big|_{x=0,z=0}\right]dy
      -\int_0^1\int_0^1\left[H_x\frac{\pa^2E_y}{\pa z\pa x}\right]_{x=0}dydz\\
   &&-\int_0^1\left[H_x\frac{\pa E_z}{\pa x}\Big|_{x=0,y=1}
      -H_x\frac{\pa E_z}{\pa x}\Big|_{x=0,y=0}\right]dz
      -\int_0^1\int_0^1\left[H_x\frac{\pa^2E_z}{\pa y\pa x}\right]_{x=0}dydz\\
   &=&0,
\enn
where use has been made of the fact that $H_x=0$ for $x=0$ and for all $t>0$, $y\in\Om$
in view of the PEC boundary condition $\vec{n}\cdot\Hb=0$ on $\pa\Om$ in (\ref{pec2.1}).

Similarly, it can be shown that $T_1=0$.
Thus (\ref{thm2.13a}) implies (\ref{thm2.1}) with $w=x$.
The theorem is thus proved.
\end{proof}

It is easy to see that the above proof of Theorem \ref{t2.1} still works
if $\E$ and $\Hb$ are replaced by ${\pa\E}/{\pa t}$ and ${\pa\Hb}/{\pa t}$,
respectively, so we have the following result.

\begin{theorem}(Energy conservation II)\label{t2.2}
Let $\E$ and $\Hb$ be the solution of the Maxwell's equations (\ref{mw2.1})-(\ref{mw2.6})
in a lossless and isotropic medium satisfying the PEC boundary condition $(\ref{pec2.1})$.
Then for $w=x,y,z$,
%\small
\be\label{thm2.2}
\int_\Om\left(\vep\left|\frac{\pa^2\E}{\pa w\pa t}\right|^2
    +\mu\left|\frac{\pa^2\Hb}{\pa w\pa t}\right|^2\right)dxdydz\equiv\mbox{Constant}.
\en
%\normalsize
\end{theorem}

As far as we know, the two energy conservation laws (\ref{thm2.1}) and (\ref{thm2.2})
above are new. However, the following two energy conservation laws are already known
and were proved in \cite{chen}; in particular, the first one is the well-known Poynting
theorem and can be found in many classical physical books.

\begin{lemma}(Energy conservation III)\label{l2.1}
Let $\E$ and $\Hb$ be the solution of the Maxwell's equations (\ref{mw2.1})-(\ref{mw2.6})
in a lossless medium satisfying the PEC boundary condition (\ref{pec2.1}).
Then
%\small
\be\label{lem2.1}
\int_\Om\left(\vep|\E|^2+\mu|\Hb|^2\right)dxdydz\equiv\mbox{Constant}
\en
\end{lemma}
%\normalsize

\begin{lemma}(Energy conservation IV)\label{l2.2}
Let $\E$ and $\Hb$ be the solution of the Maxwell's equations (\ref{mw2.1})-(\ref{mw2.6})
in a lossless medium satisfying the PEC boundary condition (\ref{pec2.1}).
Then
%\small
\be\label{lem2.2}
\int_\Om\left(\vep\left|\frac{\pa\E}{\pa t}\right|^2
    +\mu\left|\frac{\pa\Hb}{\pa t}\right|^2\right)dxdydz\equiv\mbox{Constant}
\en
\end{lemma}
%\normalsize

We now have four conservation laws (\ref{thm2.1}), (\ref{thm2.2}), (\ref{lem2.1})
and (\ref{lem2.2}) for the propagation of electromagnetic waves in a lossless
and isotropic medium. We will prove in the next section that the ADI-FDTD scheme
keeps these properties with a second-order in time perturbation term.

\section{Energy identities of the ADI-FDTD scheme in $H^1$}\label{dis-energy}

Before establishing the energy identities for the
ADI-FDTD scheme, we first recall the ADI-FDTD scheme (see \cite{Zheng1}) and
define some discrete energy norms.

\subsection{The ADI-FDTD scheme}\label{sec3.1}

The ADI-FDTD scheme makes use of Yee's spatial staggered grids (see \cite{Yee}),
which are denoted by $\Om^h$ and defined as:
%\small
\ben
\Om^h&=:&\{(x_\al,y_{\beta},z_\g)\;|\;x_\al=\al\dx,y_\beta=\beta\dy,z_\g=\g\dz,\,\al=i,i+1/2,\\
     &&\;i=0,1\cdots,I-1;\;\beta=j,j+1/2,\;j=0,\cdots,J-1;\,\g=k,k+1/2,\\
     &&\;k=0,\cdots,K-1;\;x_0=y_0=z_0=0,\,x_I=y_J=z_K=1\},
\enn
%\normalsize
where $\Delta x$, $\Delta y$ and $\Delta z$ are the mesh sizes along the $x$, $y$ and $z$
directions, respectively. In the ADI-FDTD scheme, the field components are defined on
different subsets of $\Om^h$ (see the subscripts in the equations of this scheme below).
%$E_x$ on $\{(x_{i+{1}/{2}},y_j,z_k)\}$, $E_y$ on $\{(x_i,y_{j+{1}/{2}},z_k)\}$,
%$H_x$ on $\{(x_i,y_{j+{1}/{2}},z_{k+{1}/{2}})\}$, $H_y$ on $\{(x_{i+{1}/{2}},y_j,z_{k+{1}/{2}})\}$,
%$H_z$ on $\{(x_{i+{1}/{2}},y_{k+{1}/{2}},z_k)\}$.

For a positive integer $N$ let $\Delta t=T/N$ be the time step and define $t^n=:n\dt$
with $n=0,1,\cdots,N.$ Denote by $u^m_{\al,\beta,\g}$ the grid function defined on the point
$(m\dt,\al\dx,\beta\dy,\g\dz)$. Then define
%\small
\ben
&&\delta_xu^m_{\al,\beta,\g}=({u^m_{\al+\frac12,\beta,\g}-u^m_{\al-\frac12,\beta,\g}})/{\dx},\;\;\;
  \delta_tu^m_{\al,\beta,\g}=({u^{m+\frac12}_{\al,\beta,\g}-u^{m-\frac12}_{\al,\beta,\g}})/{\dt},\\
&&\bar{\delta_t}u^m_{\al,\beta,\g}=({u^{m+\frac12}_{\al,\beta,\g}+u^{m-\frac12}_{\al,\beta,\g}})/{2}.
\enn
%\normalsize
$\delta_yu^m_{\al,\beta,\g}$ and $\delta_zu^m_{\al,\beta,\g}$ can be defined similarly.
%\ben
%\delta_xu^m_{\al,\beta,\g}&=&\frac{u^m_{\al+\frac12,\beta,\g}-u^m_{\al-\frac12,\beta,\g}}{\dx},\qquad
%\delta_yu^m_{\al,\beta,\g}=\frac{u^m_{\al,\beta+\frac12,\g}-u^m_{\al,\beta-\frac12,\g}}{\dy}\\
%\delta_zu^m_{\al,\beta,\g}&=&\frac{u^m_{\al,\beta,\g+\frac12}-u^m_{\al,\beta,\g-\frac12}}{\dz},\qquad
%\delta_tu^m_{\al,\beta,\g}=\frac{u^{m+\frac12}_{\al,\beta,\g}-u^{m-\frac12}_{\al,\beta,\g}}{\dt},\\
%\bar{\delta_t}u^m_{\al,\beta,\g}&=&\frac{u^{m+\frac12}_{\al,\beta,\g}+u^{m-\frac12}_{\al,\beta,\g}}{2}.
%\enn \normalsize
Denote by ${E_w^m}_{\al,\beta,\g}$ and ${H_w^m}_{\al,\beta,\g}$ the approximations of the electric field
$E_w(t^m,x_\al,y_\beta,z_\g)$ and the magnetic field $H_w(t^m,x_\al,y_\beta,z_\g)$, $w=x,y,z.$
Then the ADI-FDTD scheme (see \cite{Zheng1}) is defined in the following two stages.\\
{\sf Stage 1:\hfill}
\small
\be\label{zh1.1}
&&\frac{E_x^{\bar{n}}-E_x^n}{\dt}=\frac{1}{2\vep}\Big(\delta_yH_z^{\bar{n}}
    -\delta_zH_y^n\Big)\Big|_{\bar{i},j,k},\;\; %\\ \label{zh1.2}
\frac{E_y^{\bar{n}}-E_y^n}{\dt}=\frac{1}{2\vep}\Big(\delta_zH_x^{\bar{n}}
    -\delta_xH_z^n\Big)\Big|_{i,\bar{j},k},\\ \label{zh1.3}
&&\frac{E_z^{\bar{n}}-E_z^n}{\dt}=\frac{1}{2\vep}\Big(\delta_xH_y^{\bar{n}}
    -\delta_yH_x^n\Big)\Big|_{i,j,\bar{k}}, \;\; %\\ \label{zh1.4}
\frac{H_x^{\bar{n}}-H_x^n}{\dt}=\frac{1}{2\mu}\Big(\delta_zE_y^{\bar{n}}
    -\delta_yE_z^n\Big)\Big|_{i,\bar{j},\bar{k}},\\ \label{zh1.5}
&&\frac{H_y^{\bar{n}}-H_y^n}{\dt}=\frac{1}{2\mu}\Big(\delta_xE_z^{\bar{n}}
    -\delta_zE_x^n\Big)\Big|_{\bar{i},j,\bar{k}}, \;\; %\\ \label{zh1.6}
\frac{H_z^{\bar{n}}-H_z^n}{\dt}=\frac{1}{2\mu}\Big(\delta_yE_x^{\bar{n}}
    -\delta_xE_y^n\Big)\Big|_{\bar{i},\bar{j},k},
\en
\normalsize
whereafter $\bar{l}:=l+1/2$ for $l=n,i,j,k$ and $F|_{\al,\beta,\g}$ means each term of
the expression or the equation $F$ has the subscripts $\al,\beta,\g.$\\
{\sf Stage 2:\hfill}
\small
\be\label{zh2.1}
&&\qquad\frac{E_x^{n+1}-E_x^{\bar{n}}}{\dt}=\frac{1}{2\vep}\Big(\delta_yH_z^{\bar{n}}
     -\delta_zH_y^{n+1}\Big)\Big|_{\bar{i},j,k},\;\; %\\ \label{zh2.2}
\frac{E_y^{n+1}-E_y^{\bar{n}}}{\dt}=\frac{1}{2\vep}\Big(\delta_zH_x^{\bar{n}}
    -\delta_xH_z^{n+1}\Big)\Big|_{i,\bar{j},k},\\ \label{zh2.3}
&&\qquad\frac{E_z^{n+1}-E_z^{\bar{n}}}{\dt}=\frac{1}{2\vep}\Big(\delta_xH_y^{\bar{n}}
    -\delta_yH_x^{n+1}\Big)\Big|_{i,j,\bar{k}}, \;\; %\\ \label{zh2.4}
\frac{H_x^{n+1}-H_x^{\bar{n}}}{\dt}=\frac{1}{2\mu}\Big(\delta_zE_y^{\bar{n}}
    -\delta_yE_z^{n+1}\Big)\Big|_{i,\bar{j},\bar{k}},\\ %\label{zh2.5}
&&\qquad\frac{H_y^{n+1}-H_y^{\bar{n}}}{\dt}=\frac{1}{2\mu}\Big(\delta_xE_z^{\bar{n}}
   -\delta_zE_x^{n+1}\Big)\Big|_{\bar{i},j,\bar{k}}, \;\; %\\ \label{zh2.6}
\frac{H_z^{n+1}-H_z^{\bar{n}}}{\dt}=\frac{1}{2\mu}\Big(\delta_yE_x^{\bar{n}}
   -\delta_xE_y^{n+1}\Big)\Big|_{\bar{i},\bar{j},k}. \label{zh2.6}
\en
\normalsize

The PEC boundary condition (\ref{pec2.1}) can be discretized as:
%\small
\be\no
&&E^m_{x_{i+\frac12,0,k}}=E^m_{x_{i+\frac12,J,k}}=E^m_{x_{i+\frac12,j,0}}
   =E^m_{x_{i+\frac12,j,K}}=0,\\ \label{pec2}
&&E^m_{y_{0,j+\frac12,k}}=E^m_{y_{I,j+\frac12,k}}=E^m_{y_{i,j+\frac12,0}}
   =E^m_{y_{i,j+\frac12,K}}=0,\\ \no
&&E^m_{z_{0,j,k+\frac12}}=E^m_{z_{I,j,k+\frac12}}=E^m_{z_{i,0,k+\frac12}}
   =E^m_{z_{i,J,k+\frac12}}=0,
\en
%\normalsize
where $m=n$ or $n+1/2$ denotes the time levels and $i,j,k$ are integers in their valid ranges.
We also need the discrete initial conditions which are obtained by imposing the initial
condition (\ref{ic}) at $t=0$:
\ben\no
\E^0_{\al,\beta,\g}=\E_0(\al\dx,\beta\dy,\g\dz),\qquad
\Hb^0_{\al,\beta,\g}=\Hb_0(\al\dx,\beta\dy,\g\dz).
\enn

\subsection{Discrete energy norms and notations}\label{sec3.2}

We will need some discrete energy norms and notations.
For $\ds{\bf U}=({U_x}_{\bar{i},j,k},{U_y}_{i,\bar{j},k},{U_z}_{i,j,\bar{k}})$,
$\ds{\bf V}=({V_x}_{i,\bar{j},\bar{k}},{V_y}_{\bar{i},j,\bar{k}},{V_z}_{\bar{i},\bar{j},k})$,
define the discrete norms $\|\cdot\|_E$ and $\|\cdot\|_H$ as:
\small
\ben
&&\|{\bf U}\|_E^2=\Big[\sum\limits_{i=0}^{I-1}\sum\limits_{j=1}^{J-1}\sum\limits_{k=1}^{K-1}
    P({\bf U}){U^2_x}_{\bar{i},j,k}
  +\sum\limits_{i=1}^{I-1}\sum\limits_{j=0}^{J-1}\sum\limits_{k=1}^{K-1}
    P({\bf U}){U^2_y}_{i,\bar{j},k}
+\sum\limits_{i=1}^{I-1}\sum\limits_{j=1}^{J-1}\sum\limits_{k=0}^{K-1}
    P({\bf U}){U^2_z}_{i,j,\bar{k}}\Big]\Delta v\\
&&\|{\bf V}\|^2_H=\Big[\sum\limits_{i=0}^{I}\sum\limits_{j=0}^{J-1}\sum\limits_{k=0}^{K-1}
    P({\bf V}){V^2_x}_{i,\bar{j},\bar{k}}
  +\sum\limits_{i=0}^{I-1}\sum\limits_{j=0}^J\sum\limits_{k=0}^{K-1}
    P({\bf V}){V^2_y}_{\bar{i},j,\bar{k}}
+\sum\limits_{i=0}^{I-1}\sum\limits_{j=0}^{J-1}\sum\limits_{k=0}^K
    P({\bf V}){V^2_z}_{\bar{i},\bar{j},k}\Big]\Delta v
\enn
\normalsize
where $\Delta v=\dx\dy\dz,$ $P(\E)=\vep$, the electric permittivity, and $P(\Hb)=\mu$,
the magnetic permeability.
For the electric field $\E$ and the magnetic field $\Hb$ we also need the following norms
near the boundary grids:
\small
\ben
\|{\bf E}\|^2_I&=&\frac{1}{\dx}\left[\sum\limits_{j=0}^{J-1}\sum\limits_{k=1}^{K-1}
                \sum\limits_{i'=1,I-1}\vep E_y^2\big|_{i',\bar{j},k}
    +\sum\limits_{j=1}^{J-1}\sum\limits_{k=0}^{K-1}
     \sum\limits_{i'=1,I-1}\vep E_z^2\big|_{i',j,\bar{k}}\right]\dy\dz,\\
\|{\bf H}\|^2_{I}&=&\frac{1}{\dx}\sum\limits_{j=0}^{J-1}\sum\limits_{k=0}^{K-1}\mu
     \Big[{H^2_x}\big|_{1,\bar{j},\bar{k}}+{H^2_x}\big|_{I-1,\bar{j},\bar{k}}\Big]\dy\dz.
\enn
%\ben
%\|{\bf E}\|^2_{I}&=&\frac{1}{\dx}\sum\limits_{j=0}^{J-1}\sum\limits_{k=1}^{K-1}\vep
%     \Big[({E_y}_{1,j+\frac{1}{2},k})^2+({E_y}_{I-1,j+\frac{1}{2},k})^2\Big]\dy\dz\\
%&&+\frac{1}{\dx}\sum\limits_{j=1}^{J-1}\sum\limits_{k=0}^{K-1}\vep
%     \Big[({E_z}_{1,j,k+\frac{1}{2}})^2+({E_z}_{I-1,j,k+\frac{1}{2}})^2\Big]\dy\dz,\\
%\|{\bf H}\|^2_{I}&=&\frac{1}{\dx}\sum\limits_{j=0}^{J-1}\sum\limits_{k=0}^{K-1}\mu
%     \Big[({H_x}_{1,j+\frac12,k+\frac12})^2+({H_x}_{I-1,j+\frac12,k+\frac12})^2\Big]\dy\dz.
%\enn
\normalsize
The norms $\|\E\|_{J}$, $\|\E\|_{K}$, $\|\Hb\|_{J}$ and $\|\Hb\|_{K}$ are similarly defined
by changing the indices $j,k$, the factor $1/\dx$ and the area element $\dy\dz$ into the
corresponding ones.

Furthermore, for simplicity in notations, we introduce some difference operators.
For a vector-valued grid function ${\bf U}=({U_x}_{\al,\beta,\g},{U_y}_{\al,\beta,\g},
{U_z}_{\al,\beta,\g})$ with $\al=i$ or $i+1/2$, $\beta=j$ or $j+1/2$, $\g=k$ or $k+1/2$
define
%\small
\ben
&&\delta_w{\bf U}=(\delta_w{U_x},\delta_w{U_y},
      \delta_w{U_z})|_{\al,\beta,\g},\;\;w=x,y,z,\\
&&{\bf\delta}_1^h{\bf U}=(\delta_y{U_z},\delta_z{U_x},\delta_x{U_y})|_{\al,\beta,\g},\;\;
  {\bf\delta}^h_2{\bf U}=(\delta_z{U_y},\delta_x{U_z},\delta_y{U_x})|_{\al,\beta,\g}.
%\nabla^h\cdot{\bf W}&=&\delta_x{W_x}_{\al,\beta,\g}+\delta_y{W_y}_{\al,\beta,\g}
%       +\delta_z{W_z}_{\al,\beta,\g}.
\enn
%\normalsize
%Then $(\delta_1^h-\delta_2^h){\bf W}$ and $\nabla^h\cdot {\bf W}$ can be regarded as the
%discrete form of $\nabla \times {\bf U}$ and $\nabla\cdot {\bf W}$, respectively,
%where $\nabla=({\pa}/{\pa x},{\pa}/{\pa y},{\pa}/{\pa z})$ is the gradient operator.
By composition of operators it can be easily derived that
\ben
\delta_1^h\delta_1^h{\bf U}&=&(\delta_y\delta_xU_y,\delta_z\delta_yU_z,
           \delta_x\delta_zU_x)|_{\al,\beta,\g},\;\\
\delta^h_2\delta_2^h{\bf U}&=&(\delta_z\delta_x{U_z},\delta_x\delta_y{U_x},
           \delta_y\delta_z{U_y})|_{\al,\beta,\g}.
\enn

\subsection{Four discrete energy identities for the ADI-FDTD scheme}\label{sec3.3}

We now derive four discrete energy identities for the ADI-FDTD scheme, which are second-order
in time perturbations of the four energy conservation laws for the Maxwell equations.
To this end, we need the following result on summation by parts.

\begin{lemma}(Summation by Parts)\label{l3.1}
Let $M\ge 1$ be an integer. For any sequence $\{W_{m+1/2}\}_{m=0}^{M-1}$ let
$\delta$ is a difference operator defined by $\delta W_m=(W_{m+1/2}-W_{m-1/2})/h$,
where $h$ is a positive number. Then, for any two sequences $\{U_{m+1/2}\}^{M-1}_{m=0}$
and $\{V_m\}_{m=0}^M$ we have
\be\label{l3.1a}
\ds\sum_{m=0}^{M-1}U_{m+1/2}\delta V_{m+1/2}=U_{M-1/2}V_M-U_{1/2}V_0
  -\sum_{m=1}^{M-1}V_m\delta U_m.
\en
\end{lemma}

\begin{theorem}(Energy identity I)\label{t3.1}
Let $n\ge0$ and let $\E^n=({E_x^n}_{\bar{i},j,k},{E_y^n}_{i,\bar{j},k},{E_z^n}_{i,j,\bar{k}})$,
$\Hb^n=({H_x^n}_{i,\bar{j},\bar{k}},{H_y^n}_{\bar{i},j,\bar{k}},{H_z^n}_{\bar{i},\bar{j},k})$
be the solution of the ADI-FDTD scheme $(\ref{zh1.1})-(\ref{zh2.6})$ with the boundary
condition $(\ref{pec2})$. Then we have the following energy identities:
%\small
\be\no
&&\|\delta_w\E^{n+1}\|^2_E+\|\delta_w\Hb^{n+1}\|^2_H
  +\frac{(\dt)^2}{4\mu\vep}\Big(\|\delta_w\delta_1^h\E^{n+1}\|^2_H
  +\|\delta_w\delta_2^h\Hb^{n+1}\|^2_E\Big)\\ \no
&&\quad\quad\;\;+\|\E^{n+1}\|^2_{L(w)}+\|\Hb^{n+1}\|^2_{L(w)}
  +\frac{(\dt)^2}{4\mu\vep}\Big(\|\delta_2^h{\bf H}^{n+1}\|^2_{L(w)}
  +\|\delta_1^h{\bf E}^{n+1}\|^2_{L(w)}\Big)\\ \no
&&\quad\;\;=\|\delta_w\E^{n}\|^2_E+\|\delta_w\Hb^{n}\|^2_H
  +\frac{(\dt)^2}{4\mu\vep}\Big(\|\delta_w\delta_1^h\E^{n}\|^2_H
  +\|\delta_w\delta_2^h\Hb^{n}\|^2_E\Big)\\ \label{t3.1a}
&&\quad\quad\;\;+\|\E^{n}\|^2_{L(w)}+\|\Hb^{n}\|^2_{L(w)}
  +\frac{(\dt)^2}{4\mu\vep}\Big(\|\delta_2^h{\bf H}^{n}\|^2_{L(w)}
  +\|\delta_1^h{\bf E}^{n}\|^2_{L(w)}\Big),
\en
%\normalsize
where $w=x,y,z,\,L(x)=I,\,L(y)=J,\,L(z)=K$.
\end{theorem}

\begin{proof}
We only prove (\ref{t3.1a}) for the case with $w=x.$ The other cases with $w=y,z$ can be
proved similarly.

Applying the difference operator $\delta_x$ to each equation in Stages 1 of the
ADI-FDTD scheme, %we have\\
%{\sf $\delta_x$-Stage 1:\hfill}
%\ben\no
%&&\delta_x{E_x^{n+\frac12}}-\frac{\dt}{2\vep}\delta_x\delta_y{H_z^{n+\frac12}}
%   =\delta_x{E_x^n}-\frac{\dt}{2\vep}\delta_x\delta_z{H_y^n}\Big|_{i,j,k},\\ \no
%&&\delta_x{E_y^{n+\frac12}}-\frac{\dt}{2\vep}\delta_x\delta_z{H_x^{n+\frac12}}
%   =\delta_x{E_y^n}-\frac{\dt}{2\vep}\delta_x\delta_x{H_z^n}\Big|_{i+\frac12,j+\frac12,k},\\ \no
%&&\delta_x{E_z^{n+\frac12}}-\frac{\dt}{2\vep}\delta_x\delta_x{H_y^{n+\frac12}}
%   =\delta_x{E_z^n}-\frac{\dt}{2\vep}\delta_x\delta_y{H_x^n}\Big|_{i+\frac12,j,k+\frac12},\\ \label{t3.12}
%&&\delta_x{H_x^{n+\frac12}}-\frac{\dt}{2\mu}\delta_x\delta_z{E_y^{n+\frac12}}
%   =\delta_x{H_x^n}-\frac{\dt}{2\mu}\delta_x\delta_y{E_z^n}\Big|_{i+\frac12,j+\frac12,k+\frac12},\\ \no
%&&\delta_x{H_y^{n+\frac12}}-\frac{\dt}{2\mu}\delta_x\delta_x{E_z^{n+\frac12}}
%   =\delta_x{H_y^n}-\frac{\dt}{2\mu}\delta_x\delta_z{E_x^n}\Big|_{i,j,k+\frac12},\\ \no
%&&\delta_x{H_z^{n+\frac12}}-\frac{\dt}{2\mu}\delta_x\delta_y{E_x^{n+\frac12}}
%   =\delta_x{H_z^n}-\frac{\dt}{2\mu}\delta_x\delta_x{E_y^n}\Big|_{i,j+\frac12,k},
%\enn
%The equations in $\delta_x$-Stage 2 are similar to those in $\delta_x$-Stage 1 and
%omitted here for simplicity.
%Multiplying the first three equations in {\sf $\delta_x$-Stage 1} by $\sqrt{\vep}$ and the last three
%equations in {\sf $\delta_x$-Stage 1} by $\sqrt{\mu}$,
rearranging the terms according to the time levels
and squaring both sides of the six equations thus obtained, we have
%\small
\be\label{3.1a}
&&\vep(\delta_x{E_x^{n+\frac12}})^2+\frac{(\dt)^2}{4\vep}(\delta_x\delta_y{H_z^{n+\frac12}})^2
   -\dt\delta_x{E_x^{n+\frac12}}\delta_x\delta_y{H_z^{n+\frac12}}\\ \no
&&\qquad=\vep(\delta_x{E_x^n})^2+\frac{(\dt)^2}{4\vep}(\delta_x\delta_z{H_y^n})^2
   -\dt\delta_x{E_x^n}\delta_x\delta_z{H_y^n}\Big|_{i,j,k},\\ \label{3.1b}
&&\vep(\delta_x{E_y^{n+\frac12}})^2+\frac{(\dt)^2}{4\vep}(\delta_x\delta_z{H_x^{n+\frac12}})^2
   -\dt\delta_x{E_y^{n+\frac12}}\delta_x\delta_z{H_x^{n+\frac12}}\\ \no
&&\qquad=\vep(\delta_x{E_y^n})^2+\frac{(\dt)^2}{4\vep}(\delta_x\delta_x{H_z^n})^2
   -\dt\delta_x{E_y^n}\delta_x\delta_x{H_z^n}\Big|_{i+\frac12,j+\frac12,k},\\ \label{3.1c}
&&\vep(\delta_x{E_z^{n+\frac12}})^2+\frac{(\dt)^2}{4\vep}(\delta_x\delta_x{H_y^{n+\frac12}})^2
   -\dt\delta_x{E_z^{n+\frac12}}\delta_x\delta_x{H_y^{n+\frac12}}\\ \no
&&\qquad=\vep(\delta_x{E_z^n})^2+\frac{(\dt)^2}{4\vep}(\delta_x\delta_y{H_x^n})^2
   -\dt\delta_x{E_z^n}\delta_x\delta_y{H_x^n}\Big|_{i+\frac12,j,k+\frac12},\\ \label{3.1d}
&&\mu(\delta_x{H_x^{n+\frac12}})^2+\frac{(\dt)^2}{4\mu}(\delta_x\delta_z{E_y^{n+\frac12}})^2
   -\dt\delta_x{H_x^{n+\frac12}}\delta_x\delta_z{E_y^{n+\frac12}}\\ \no
&&\qquad=\mu(\delta_x{H_x^n})^2+\frac{(\dt)^2}{4\mu}(\delta_x\delta_y{E_z^n})^2
   -\dt\delta_x{H_x^n}\delta_x\delta_y{E_z^n}\Big|_{i+\frac12,j+\frac12,k+\frac12},\\ \label{3.1e}
&&\mu(\delta_x{H_y^{n+\frac12}})^2+\frac{(\dt)^2}{4\mu}(\delta_x\delta_x{E_z^{n+\frac12}})^2
   -\dt\delta_x{H_y^{n+\frac12}}\delta_x\delta_x{E_z^{n+\frac12}}\\ \no
&&\qquad=\mu(\delta_x{H_y^n})^2+\frac{(\dt)^2}{4\mu}(\delta_x\delta_z{E_x^n})^2
   -\dt\delta_x{H_y^n}\delta_x\delta_z{E_x^n}\Big|_{i,j,k+\frac12},\\ \label{3.1f}
&&\mu(\delta_x{H_z^{n+\frac12}})^2+\frac{(\dt)^2}{4\mu}(\delta_x\delta_y{E_x^{n+\frac12}})^2
   -\dt\delta_x{H_z^{n+\frac12}}\delta_x\delta_y{E_x^{n+\frac12}}\\ \no
&&\qquad=\mu(\delta_x{H_z^n})^2+\frac{(\dt)^2}{4\mu}(\delta_x\delta_x{E_y^n})^2
   -\dt\delta_x{H_z^n}\delta_x\delta_x{E_y^n}\Big|_{i,j+\frac12,k}.
\en
%\normalsize
From the definition of $\delta_x$ and the conditions (\ref{pec2}) it follows that
%\small
\be\label{pec3}
\delta_xE^m_{x_{i,j',k}}=\delta_xE^m_{x_{i,j,k'}}=\delta_xE^m_{y_{i+\frac12,j+\frac12,k'}}
=\delta_xE^m_{z_{i+\frac12,j',k+\frac12}}=0
\en
%\normalsize
where $m=n,\,n+1/2,$ $j'=0,\,J,\,k'=0,\,K.$ By these boundary conditions and summation
by parts (Lemma \ref{l3.1}), we can see that the sum of all the mixed-product terms such as
$\delta_xE_u\delta_x\delta_uH_v\;(u,v=x,y,z)$ on the left- and right-hand sides of
(\ref{3.1a})-(\ref{3.1f}) will vanish, respectively, due to cancelation.
For example, consider the sum of the mixed-product terms on the left hand sides of (\ref{3.1f}),
(\ref{3.1d}) and (\ref{3.1e}) over $i$, $j$ and $k$ in their valid range. We have
\small
\ben
&&\sum\limits_{i=1}^{I-1}\sum\limits_{j=0}^{J-1}\sum\limits_{k=1}^{K-1}
    \delta_x{H_z^{\bar{n}}}\delta_x\delta_y{E_x^{\bar{n}}}\Big|_{i,j+\frac12,k}
   =-\sum\limits_{i=1}^{I-1}\sum\limits_{j=1}^{J-1}\sum\limits_{k=1}^{K-1}
    \delta_x{E_x^{\bar{n}}}\delta_y\delta_x{H_z^{\bar{n}}}\Big|_{i,j,k},\\
&&\sum\limits_{i=1}^{I-2}\sum\limits_{j=0}^{J-1}\sum\limits_{k=0}^{K-1}
    \delta_x{H_x^{\bar{n}}}\delta_x\delta_z{E_y^{\bar{n}}}\Big|_{i+\frac12,j+\frac12,k+\frac12}
  =-\sum\limits_{i=1}^{I-2}\sum\limits_{j=0}^{J-1}\sum\limits_{k=1}^{K-1}
    \delta_x{E_y^{\bar{n}}}\delta_z\delta_x{H_x^{\bar{n}}}\Big|_{i+\frac12,j+\frac12,k},\\
&&\sum\limits_{i=1}^{I-1}\sum\limits_{j=1}^{J-1}\sum\limits_{k=0}^{K-1}
    \delta_x{H_y^{\bar{n}}}\delta_x\delta_x{E_z^{\bar{n}}}\Big|_{i,j,k+\frac12}
=-\sum\limits_{i=1}^{I-2}\sum\limits_{j=1}^{J-1}\sum\limits_{k=1}^{K-1}
    \delta_x{E_z^{\bar{n}}}\delta_x\delta_x{H_y^{\bar{n}}}\Big|_{i+\frac12,j,k+\frac12}\\
&&\qquad+\frac{1}{\dx}\sum\limits_{j=1}^{J-1}\sum\limits_{k=0}^{K-1}
    \Big\{\delta_x{E_z}^{\bar{n}}_{I-\frac12,j,k+\frac12}
    \delta_x{H_y}^{\bar{n}}_{I-1,j,k+\frac12}
    -\delta_x{E_z}^{\bar{n}}_{\frac12,j,k+\frac12}
    \delta_x{H_y}^{\bar{n}}_{1,j,k+\frac12,k+\frac{1}{2}}\Big\}.
\enn
%\ben
%&&\sum\limits_{i=1}^{I-1}\sum\limits_{j=0}^{J-1}\sum\limits_{k=1}^{K-1}
%    \delta_x{H_z^{n+\frac12}}\delta_x\delta_y{E_x^{n+\frac12}}\Big|_{i,j+\frac12,k}
%   =-\sum\limits_{i=1}^{I-1}\sum\limits_{j=1}^{J-1}\sum\limits_{k=1}^{K-1}
%    \delta_x{E_x^{n+\frac12}}\delta_y\delta_x{H_z^{n+\frac12}}\Big|_{i,j,k},\\
%&&\sum\limits_{i=1}^{I-2}\sum\limits_{j=0}^{J-1}\sum\limits_{k=0}^{K-1}
%    \delta_x{H_x^{n+\frac12}}\delta_x\delta_z{E_y^{n+\frac12}}\Big|_{i+\frac12,j+\frac12,k+\frac12}\\
%&&\qquad\qquad\qquad\qquad =-\sum\limits_{i=1}^{I-2}\sum\limits_{j=0}^{J-1}\sum\limits_{k=1}^{K-1}
%    \delta_x{E_y^{n+\frac12}}\delta_z\delta_x{H_x^{n+\frac12}}\Big|_{i+\frac12,j+\frac12,k},\\
%&&\sum\limits_{i=1}^{I-1}\sum\limits_{j=1}^{J-1}\sum\limits_{k=0}^{K-1}
%    \delta_x{H_y^{n+\frac12}}\delta_x\delta_x{E_z^{n+\frac12}}\Big|_{i,j,k+\frac12}
%=-\sum\limits_{i=1}^{I-2}\sum\limits_{j=1}^{J-1}\sum\limits_{k=1}^{K-1}
%    \delta_x{E_z^{n+\frac12}}\delta_x\delta_x{H_y^{n+\frac12}}\Big|_{i+\frac12,j,k+\frac12}\\
%&&\qquad+\frac{1}{\dx}\sum\limits_{j=1}^{J-1}\sum\limits_{k=0}^{K-1}
%    \Big\{\delta_x{E_z}^{n+\frac12}_{I-\frac12,j,k+\frac12}
%    \delta_x{H_y}^{n+\frac12}_{I-1,j,k+\frac12}
%%&&\qquad\qquad\qquad
%    -\delta_x{E_z}^{n+\frac12}_{\frac12,j,k+\frac12}
%    \delta_x{H_y}^{n+\frac12}_{1,j,k+\frac12,K}\Big\}.
%\enn
\normalsize
The sum of the above three terms can cancel the sums of the three mixed-product terms on
the left-hand side of (\ref{3.1a}), (\ref{3.1b}) and (\ref{3.1c}), respectively.
Thus, sum up each of the six equations above over the valid ranges of their subscripts $i,j,k$
and add the updated six equations together to deduce that
\small
\be\no
&&\|\delta_x{\bf E}^{n+\frac12}\|^2_E+\|\delta_x{\bf H}^{n+\frac12}\|^2_H
    +\frac{(\dt)^2}{4\mu\vep}(\|\delta_x\delta_2^h{\bf E}^{n+\frac12}\|^2_H
    +\|\delta_x\delta_1^h{\bf H}^{n+\frac12}\|_E^2)\\ \no
&&\qquad=\|\delta_x{\bf E}^n\|^2_E+\|\delta_x{\bf H}^n\|^2_H
  +\frac{(\dt)^2}{4\mu\vep}(\|\delta_x\delta_1^h{\bf E}^n\|^2_H
  +\|\delta_x\delta_2^h{\bf H}^n\|_E^2)\\ \no
&&\qquad\quad-\frac{\dt}{\dx}\sum\limits_{j=1}^{J-1}\sum\limits_{k=0}^{K-1}
   \Big[{E_z^{n+\frac12}}\delta_x{H_y^{n+\frac12}}\Big|_{I-1,j,k+\frac12}
   +{E_z^{n+\frac12}}\delta_x{H_y^{n+\frac12}}\Big|_{1,j,k+\frac12}\Big]\dy\dz\\ \label{3.2}
&&\qquad\quad+\frac{\dt}{\dx}\sum\limits_{j=0}^{J-1}\sum\limits_{k=1}^{K-1}
   \Big[{E_y^n}\delta_x{H_z^n}\Big|_{I-1,j+\frac12,k}
   +{E_y^n}\delta_x{H_z^n}\Big|_{1,j+\frac12,k}\Big]\dy\dz,
\en
\normalsize
where we have used the boundary conditions (\ref{pec2}) applied by the operator $\delta_x$
to get the last two terms on the right-hand side of the above equation.
%while applying the operator on the terms at the points $(x_{i^\prime},y_j,z_{k+\frac{1}{2}})$
%with $i^\prime=1,I-1$.

Similarly, from the equations in Stage 2 applied by the operator $\delta_x$ we can derive that
\small
\be\no
&&\|\delta_x{\bf E}^{n+1}\|_E^2+\|\delta_x{\bf H}^{n+1}\|_H^2
   +\frac{(\dt)^2}{4\mu\vep}\Big(\|\delta_x{\bf \delta}_1^h{\bf E}^{n+1}\|_H^2
   +\|\delta_x{\bf \delta}_2^h{\bf H}^{n+1}\|_E^2\Big)\\ \no
&&\quad=\|\delta_x{\bf E}^{n+\frac{1}{2}}\|_E^2+\|\delta_x{\bf H}^{n+\frac{1}{2}}\|_H^2
   +\frac{(\dt)^2}{4\mu\vep}(\|\delta_x\delta_2^h{\bf E}^{n+\frac{1}{2}}\|_H^2
   +\|\delta_x\delta_1^h{\bf H}^{n+\frac{1}{2}}\|_E^2)\\ \no
&&\qquad\quad-\frac{\dt}{\dx}\sum\limits_{j=1}^{J-1}\sum\limits_{k=0}^{K-1}
   \Big({E_z}^{n+\frac{1}{2}}\delta_x{H_y}^{n+\frac{1}{2}}\Big|_{I-1,j,k+\frac{1}{2}}
   +{E_z}^{n+\frac{1}{2}}\delta_x{H_y}^{n+\frac{1}{2}}\Big|_{1,j,k+\frac{1}{2}}\Big)\dy\dz\\ \label{3.3}
&&\qquad\quad+\frac{\dt}{\dx}\sum\limits_{j=0}^{J-1}\sum\limits_{k=1}^{K-1}
   \Big({E_y}^{n+1}\delta_x{H_z}^{n+1}\Big|_{I-1,j+\frac{1}{2},k}
   +{E_y}^{n+1}\delta_x{H_z}^{n+1}\Big|_{1,j+\frac{1}{2},k}\Big)\dy\dz.
\en
\normalsize

We now consider the mixed-product terms on the right-hand side of (\ref{3.2}) and (\ref{3.3}).
Taking $i=1$ and $I-1$ in the first equation in (\ref{zh1.3}) and (\ref{zh2.3}) we have
%\small
\ben
{E_z}^{n+\frac{1}{2}}_{i^{\prime},j,k+\frac{1}{2}}
    -\frac{\dt}{2\vep}\delta_x{H_y}^{n+\frac{1}{2}}_{i^\prime,j,k+\frac12}
    ={E_z}^n_{i^\prime,j,k+\frac12}-\frac{\dt}{2\vep}\delta_y{H_x}^n_{i^\prime,j,k+\frac12},\\
{E_z}^{n+1}_{i^\prime,j,k+\frac12}+\frac{\dt}{2\vep}\delta_y{H_x}^{n+1}_{i^\prime,j,k+\frac12}
    ={E_z}^{n+\frac{1}{2}}_{i^\prime,j,k+\frac{1}{2}}
    +\frac{\dt}{2\vep}\delta_x{H_y}^{n+\frac12}_{i^\prime,j,k+\frac12},
\enn
%\normalsize
where $i^\prime=1,I-1.$ Multiplying both sides of the above two equations by $\sqrt{\vep}$,
squaring them and adding the resulting equations give
%\small
\be\no
2\dt E_z^{n+\frac12}\delta_xH_y^{n+\frac12}&=&\vep(E_z^{n+1})^2-\vep(E_z^{n})^2
  +\frac{(\dt)^2}{4\mu\vep}\Big[\mu(\delta_yH_x^{n+1})^2-\mu(\delta_yH_x^{n})^2\Big]\\ \label{3.4}
&&+\dt\Big(E_z^{n+1}\delta_yH_x^{n+1}+E_z^n\delta_yH_x^n\Big)\Big|_{i',j,k+\frac12},
\qquad i'=1,\,I-1.
\en
%\normalsize
Similarly, it follows from the second equation in (\ref{zh1.1}) and (\ref{zh2.1})
with $i=1,I-1$ that
%\small
\be\no
&&\dt\Big(E_y^n\delta_xH_z^n+E_y^{n+1}\delta_xH_z^{n+1}\Big)
    =\vep\big(E_y^n\big)^2-\vep\big(E_y^{n+1}\big)^2
     +2\dt E_y^{n+\frac12}\delta_zH_x^{n+\frac12}\\ \label{3.5}
&&\qquad\qquad+\frac{(\dt)^2}{4\mu\vep}\Big(\mu(\delta_xH_z^n)^2
    -\mu(\delta_xH_z^{n+1})^2\Big)\Big|_{i',j+\frac12,k},\qquad i'=1,\,I-1.
\en
%\normalsize
Combining (\ref{3.2}), (\ref{3.3}), (\ref{3.4}) and (\ref{3.5}) gives
%\small
\be\no
&&\|\delta_x{\bf E}^{n+1}\|_E^2+\|\delta_x{\bf H}^{n+1}\|^2_H
  +\frac{(\dt)^2}{4\mu\vep}(\|\delta_x\delta_1^h{\bf E}^{n+1}\|_H^2
  +\|\delta_x\delta_2^h{\bf H}^{n+1}\|^2_E)\\ \no
&&\qquad+\frac{1}{\dx}\sum\limits_{i'=1,I-1}\sum\limits_{j=1}^{J-1}\sum\limits_{k=0}^{K-1}
  \Big[\vep({E_z}^{n+1})^2+\frac{(\dt)^2}{4\mu\vep}
  \mu(\delta_y{H_x}^{n+1})^2\Big]_{i',j,k+\frac12}\dy\dz\\ \no
&&\qquad+\frac{1}{\dx}\sum\limits_{i'=1,I-1}\sum\limits_{j=0}^{J-1}\sum\limits_{k=1}^{K-1}
  \Big[\vep({E_y}^{n+1})^2+\frac{(\dt)^2}{4\mu\vep}
  \mu(\delta_x{H_z}^{n+1})^2\Big]_{i',j+\frac12,k}\dy\dz\\ \no
&&\quad=\|\delta_x{\bf E}^n\|_E^2+\|\delta_x{\bf H}^n\|_H^2
   +\frac{(\dt)^2}{4\mu\vep}(\|\delta_x\delta_1^h{\bf E}^n\|^2_H
   +\|\delta_x\delta_2^h{\bf H}^n\|^2_E)\\ \no
&&\qquad+\frac{1}{\dx}\sum\limits_{i'=1,I-1}\sum\limits_{j=1}^{J-1}\sum\limits_{k=0}^{K-1}
  \Big[\vep({E_z}^n)^2+\frac{(\dt)^2}{4\mu\vep}
  \mu(\delta_y{H_x}^n)^2\Big]_{i',j,k+\frac12}\dy\dz\\ \label{3.6}
&&\qquad+\frac{1}{\dx}\sum\limits_{i'=1,I-1}\sum\limits_{j=0}^{J-1}\sum\limits_{k=1}^{K-1}
  \Big[+\vep({E_y}^n)^2+\frac{(\dt)^2}{4\mu\vep}
  \mu(\delta_x{H_z}^n)^2\Big]_{i',j+\frac12,k}\dy\dz+R,
\en
%\normalsize
where the term $R$ is given by
%\small
\ben
&&R=-\frac{\dt}{\dx}\sum\limits_{i'=1,I-1}\sum\limits_{j=1}^{J-1}\sum\limits_{k=0}^{K-1}
  \Big[{E_z}^{n+1}\delta_y{H_x}^{n+1}+{E_z}^n\delta_y{H_x}^n\Big]_{i',j,k+\frac12}\dy\dz\\
&&\qquad\quad+\frac{2\dt}{\dx}\sum\limits_{i'=1,I-1}\sum\limits_{j=0}^{J-1}\sum\limits_{k=1}^{K-1}
  {E_y}^{n+\frac{1}{2}}\delta_z{H_x}^{n+\frac{1}{2}}\Big|_{i^{\prime},j+\frac{1}{2},k}\dy\dz.
\enn
%\normalsize
Arguing similarly as in deriving (\ref{3.4}) we deduce from the second equation in
(\ref{zh1.3}) and (\ref{zh2.3}) with $i=i'$ that
%\small
%\be\no
%&&2\dt\sum\limits_{i'=1,I-1}\sum\limits_{j=0}^{J-1}\sum\limits_{k=1}^{K-1}
%   {E_y}^{n+\frac{1}{2}}\delta_z{H_x}^{n+\frac{1}{2}}\Big|_{i',j+\frac{1}{2},k}\dy\dz\\ \no
%&&\qquad\quad-\dt\sum\limits_{i'=1,I-1}\sum\limits_{j=1}^{J-1}\sum\limits_{k=0}^{K-1}
%   \Big[{E_z}^{n+1}\delta_y{H_x}^{n+1}+{E_z}^n\delta_y{H_x}^n\Big]_{i',j,k+\frac12}\dy\dz\\ \no
\be\no
&& R=\frac{1}{\dx}\sum\limits_{i'=1,I-1}\sum\limits_{j=0}^{J-1}\sum\limits_{k=0}^{K-1}
   \Big[\mu({H_x}^n)^2+\frac{(\dt)^2}{4\mu\vep}\vep(\delta_y{E_z}^n)^2\\ \label{3.7}
&&\qquad\qquad\qquad-\big\{\mu({H_x}^{n+1})^2
   +\frac{(\dt)^2}{4\mu\vep}\vep(\delta_y{E_z}^{n+1})^2\big\}\Big]_{i',j+\frac12,k+\frac12}\dy\dz.
\en
%\normalsize
Substituting (\ref{3.7}) into (\ref{3.6}) gives the identity (\ref{t3.1a}) with $w=x$,
noting the definition of the norms $\|{\bf E}\|_I$ and $\|{\bf H}\|_I$.
The proof is thus complete.
\end{proof}

Note that the proof of Theorem \ref{t3.1} does not depend on the time levels.
Thus, if we apply the operators $\delta_t$ and $\delta_w\delta_t$ with $w=x,y,z$ to the equations
in the ADI-FDTD scheme and repeat the above argument, then we can obtain the following result.

\begin{theorem}(Energy identities II)\label{t3.2}
Let $n\geq 1$ and let $\E^n$, $\Hb^n$ be the solution of the ADI-FDTD scheme. Then
%\small
\ben
&&\|\delta_w\delta_t\E^{n+\frac{1}{2}}\|^2_E+\|\delta_w\delta_t\Hb^{n+\frac{1}{2}}\|^2_H
   +\frac{(\dt)^2}{4\mu\vep}\Big(\|\delta_w\delta_1^h\delta_t\E^{n+\frac{1}{2}}\|^2_H
   +\|\delta_w\delta_2^h\delta_t\Hb^{n+\frac{1}{2}}\|^2_E\Big)\\
&&\quad+\|\delta_t\E^{n+\frac{1}{2}}\|^2_{L(w)}+\|\delta_t\Hb^{n+\frac{1}{2}}\|^2_{L(w)}
+\frac{(\dt)^2}{4\mu\vep}\Big(\|\delta_2^h\delta_t\Hb^{n+\frac{1}{2}}\|^2_{L(w)}
  +\|\delta_1^h\delta_t\E^{n+\frac{1}{2}}\|^2_{L(w)}\Big)\\
&&\;=\|\delta_w\delta_t\E^{n-\frac{1}{2}}\|^2_E+\|\delta_w\delta_t\Hb^{n-\frac{1}{2}}\|^2_H
  +\frac{(\dt)^2}{4\mu\vep}\Big(\|\delta_w\delta_1^h\delta_t\E^{n-\frac{1}{2}}\|^2_H
  +\|\delta_w\delta_2^h\delta_t\Hb^{n-\frac{1}{2}}\|^2_E\Big)\\
&&\quad+\|\delta_t\E^{n-\frac{1}{2}}\|^2_{L(w)}+\|\delta_t\Hb^{n-\frac{1}{2}}\|^2_{L(w)}
+\frac{(\dt)^2}{4\mu\vep}\Big(\|\delta_2^h\delta_t{\bf H}^{n-\frac{1}{2}}\|^2_{L(w)}
   +\|\delta_1^h\delta_t{\bf E}^{n-\frac{1}{2}}\|^2_{L(w)}\Big),
\enn
%\normalsize
where $w=x,y,z$, $L(x)=I$, $L(y)=J$ and $L(z)=K$.
\end{theorem}

Theorems \ref{t3.1} and \ref{t3.2} are reduced to the following two results, respectively,
when $\delta_x=I$ (identity operator).

\begin{theorem}(Energy identity III)\label{t3.3}
Let $n\ge 0$ and let $\E^n$, $\Hb^n $ be the solution of the ADI-FDTD scheme.
Then
%\small
\ben
&&\|\E^{n+1}\|^2_E +\|\Hb^{n+1}\|^2_H
  +\frac{(\dt)^2}{4\mu\vep}\Big(\|\delta_2^h\Hb^{n+1}\|^2_E+\|\delta_1^h\E^{n+1}\|^2_H\Big)\\
&&\qquad\quad=\|\E^{n}\|^2_E +\|\Hb^{n}\|^2_H
  +\frac{(\dt)^2}{4\mu\vep}\Big(\|\delta_2^h\Hb^{n}\|^2_E+\|\delta_1^h\E^{n}\|^2_H\Big).
\enn
%\normalsize
\end{theorem}

\begin{theorem}(Energy identity IV)\label{t3.4}
Let $n\ge 1$ and let $\E^n$, $\Hb^n$ be the solution of the ADI-FDTD scheme.
Then the ADI-FDTD scheme satisfies the following identity:
%\small
\ben
&&\|\delta_t\E^{n+\frac{1}{2}}\|^2_E+\|\delta_t\Hb^{n+\frac{1}{2}}\|^2_H
  +\frac{(\dt)^2}{4\mu\vep}\Big(\|\delta_2^h\delta_t\Hb^{n+\frac{1}{2}}\|^2_E
  +\|\delta_1^h\delta_t\E^{n+\frac{1}{2}}\|^2_H\Big)\\
&&\quad\quad=\|\delta_t\E^{n-\frac{1}{2}}\|^2_E+\|\delta_t\Hb^{n-\frac{1}{2}}\|^2_H
  +\frac{(\dt)^2}{4\mu\vep}\Big(\|\delta_2^h\delta_t\Hb^{n-\frac{1}{2}}\|^2_E
  +\|\delta_1^h\delta_t\E^{n-\frac{1}{2}}\|^2_H\Big).
\enn
%\normalsize
\end{theorem}

\begin{corollary}\label{cor3.1}
The ADI-FDTD scheme $(\ref{zh1.1})-(\ref{zh2.6})$ with the PEC boundary condition $(\ref{pec2})$
is unconditionally stable under the new defined discrete energy norms
and under the discrete $H^1$ norm.
\end{corollary}

\section{Optimal error estimates for the ADI-FDTD scheme}\label{s-con}

In this section, we derive optimal error estimates for the ADI-FDTD scheme
under the discrete energy norms.
We need the following well-known result called the discrete Gronwall's lemma
(see \cite{Heywood}).

\begin{lemma}\label{Gronwall}
If $\{a_n\}$, $\{b_n\}$ and $\{c_n\}$ are three positive sequences with
$\{c_n\}$ being monotone such that $a_n+b_n\leq c_n+\la\sum_{i=0}^{n-1}a_i$
with $\la>0$ and $a_0+b_0\leq c_0$ then $a_n+b_n\leq c_n\exp(n\la)$ for all $n\geq 0.$
\end{lemma}

For $n\geq 0$, $w=x,y,z$ and for $\al=i,\bar{i}$, $\beta=j,\bar{j}$, $\gamma=k,\bar{k}$
let $\ds\mE^n_{w_{\alpha,\beta,\gamma}}=e_w(t^n,x_\al,y_\beta,z_\g)-E^n_{w_{\al,\beta,\gamma}}$,
$\ds\mH^n_{w_{\al,\beta,\gamma}}=h_w(t^n,x_\al,y_\beta,z_\g)-H^n_{w_{\alpha,\beta,\gamma}}$,
$\ds\me^n=(\mE^n_x,\mE^n_y,\mE^n_z)$, $\ds\mh^n=(\mH^n_x,\mH^n_y,\mH^n_z),$
where ${\bf e}=(e_x,e_y,e_z)$, ${\bf h}=(h_x,h_y,h_z)$ is the exact solution to the
Maxwell equations (\ref{mw2.1})-(\ref{mw2.6}) with the boundary condition (\ref{pec2.1})
and the initial condition (\ref{ic}) and
$\E^n=({E_x^n}_{\bar{i},j,k},{E_y^n}_{i,\bar{j},k},{E_z^n}_{i,j,\bar{k}})$,
$\Hb^n=({H_x^n}_{i,\bar{j},\bar{k}},{H_y^n}_{\bar{i},j,\bar{k}},{H_z^n}_{\bar{i},\bar{j},k})$
is the solution of the ADI-FDTD scheme.
%with the boundary condition (\ref{pec2}).

\begin{theorem}\label{t4.1}
Let the exact solution ${\bf e}$, ${\bf h}$
%be the exact solution to the Maxwell equations $(\ref{mw2.1})-(\ref{mw2.6})$
%with the boundary condition $(\ref{pec2.1})$ and the initial condition $(\ref{ic})$
satisfy that
%\small
$
{\bf e}\in C((0,T];C^4(\bar\Om))\cap C^1((0,T];C^2(\bar\Om))
  \cap C^2((0,T];C^1(\bar\Om))\cap C^3((0,T];C(\bar\Om)),
$
$
{\bf h}\in C((0,T];C^4(\bar\Om))\cap C^1((0,T];C^2(\bar\Om))
 \cap C^2((0,T];C^1(\bar\Om))\cap C^3((0,T];C(\bar\Om)).
$
%\normalsize
%and let $\E^n$ and $\Hb^n$ be the solution of the ADI-FDTD scheme.
%Set $\mE^n={\bf e}^n-\E^n$, $\mH^n={\bf h}^n-\Hb^n$.
Then for any $n\geq 0$ and any fixed $T>0$ there is constant $C$
independent of $\dt,\dx,\dy,\dz$ such that
%\small
\be\label{t4.1a}
&&\quad\|\delta_w\mE^{n+1}\|_E^2+\|\delta_w\mH^{n+1}\|^2_H
   +\frac{(\dt)^2}{4\mu\vep}(\|\delta_w\delta_1^h\mE^{n+1}\|_H^2
   +\|\delta_w\delta_2^h\mH^{n+1}\|^2_E)\\ \no
&&\qquad\quad+\|\mE^{n+1}\|^2_{L(w)}+\|\mH^{n+1}\|^2_{L(w)}
   +\frac{(\dt)^2}{4\mu\vep}\Big[\|\delta_2^h\mH^{n+1}\|^2_{L(w)}
   +\|\delta_1^h\mE^{n+1}\|^2_{L(w)}\Big]\\ \no
&&\quad\quad\leq C\{(\dt)^4+(\dx)^4+(\dy)^4+(\dz)^4\},
\en
%\normalsize
where $w=x,y,z,$ $L(x)=I,$ $L(y)=J$ and $L(z)=K$.
\end{theorem}

\begin{proof}
We only prove (\ref{t4.1a}) for the case with $w=x$. The other cases with $w=y,z$
can be proved similarly.

In order to derive the error equations for the ADI-FDTD scheme and
the $\delta_x-$ADI-FDTD scheme (i.e. the ADI-FDTD scheme applied by $\delta_x$),
we need two discrete forms of the Maxwell equations
corresponding to the two schemes.
%Let ${\bf e}=(e_x,e_y,e_z)$ and ${\bf h}=(h_x,h_y,h_z)$ be exact solution of
%the Maxwell's equations.
Denote by $e_w^{n*}$, $h_w^{n*}$ ($w=x,y,z$) the intermediate variables defined by
%\small
\ben
&&\delta_xe_x^{n*}=\frac{1}{2}\delta_x(e_x^{n+1}+e_x^n)
    +\frac{\dt}{4\vep}\delta_x\delta_z(h_y^{n+1}-h_y^n)|_{i,j,k},\\
&&\delta_xe_y^{n*}=\frac{1}{2}\delta_x(e_y^{n+1}+e_y^n)
    +\frac{\dt}{4\vep}\delta_x\delta_x(h_z^{n+1}-h_z^n)|_{i+\frac{1}{2},j+\frac{1}{2},k},\\
&&\delta_xe_z^{n*}=\frac{1}{2}\delta_x(e_z^{n+1}+e_z^n)
    +\frac{\dt}{4\vep}\delta_x\delta_y(h_x^{n+1}-h_x^n)|_{i+\frac{1}{2},j,k+\frac{1}{2}},\\
&&\delta_xh_x^{n*}=\frac{1}{2}\delta_x(h_x^{n+1}+h_x^n)
    +\frac{\dt}{4\mu}\delta_x\delta_y(e_z^{n+1}-e_z^n)|_{i+\frac{1}{2},j+\frac{1}{2},k+\frac{1}{2}},\\
&&\delta_xh_y^{n*}=\frac{1}{2}\delta_x(h_y^{n+1}+h_y^n)
     +\frac{\dt}{4\mu}\delta_x\delta_z(e_x^{n+1}-e_x^n)|_{i,j,k+\frac{1}{2}},\\
&&\delta_xh_z^{n*}=\frac{1}{2}\delta_x(h_z^{n+1}+h_z^n)
   +\frac{\dt}{4\mu}\delta_x\delta_x(e_y^{n+1}-e_y^n)|_{i,j+\frac{1}{2},k}.
\enn
%\normalsize
Then by a direct calculation we derive that\\
{\sf Stage 1 of the discrete form of the Maxwell equations:\hfill}
%\small
\be\label{Max1.1}
&&\frac{\delta_xe_x^{n*}-\delta_xe_x^n}{\dt/2}=\frac{1}{\vep}\Big(\delta_y\delta_xh_z^{n*}
    -\delta_z\delta_xh_y^n\Big)+\delta_x\beta_x^{n+\frac12}\Big|_{i,j,k},\\ \label{Max1.2}
&&\frac{\delta_xe_y^{n*}-\delta_xe_y^n}{\dt/2}=\frac{1}{\vep}\Big(\delta_z\delta_xh_x^{n*}
   -\delta_x\delta_xh_z^n\Big)+\delta_x\beta_y^{n+\frac12}\Big|_{i+\frac12,j+\frac12,k},\\ \label{Max1.3}
&&\frac{\delta_xe_z^{n*}-\delta_xe_z^n}{\dt/2}=\frac{1}{\vep}\Big(\delta_x\delta_xh_y^{n*}
    -\delta_y\delta_xh_x^n\Big)+\delta_x\beta_z^{n+\frac12}\Big|_{i+\frac12,j,k+\frac12},\\ \label{Max1.4}
&&\frac{\delta_xh_x^{n*}-\delta_xh_x^n}{\dt/2}=\frac{1}{\mu}\Big(\delta_z\delta_xe_y^{n*}
  -\delta_y\delta_xe_z^n\Big)+\delta_x\xi_x^{n+\frac12}\Big|_{i+\frac12,j+\frac12,k+\frac12},\\ \label{Max1.5}
&&\frac{\delta_xh_y^{n*}-\delta_xh_y^n}{\dt/2}=\frac{1}{\mu}\Big(\delta_x\delta_xe_z^{n*}
    -\delta_z\delta_xe_x^n\Big)+\delta_x\xi_y^{n+\frac12}\Big|_{i,j,k+\frac12},\\ \label{Max1.6}
&&\frac{\delta_xh_z^{n*}-\delta_xh_z^n}{\dt/2}=\frac{1}{\mu}\Big(\delta_y\delta_xe_x^{n*}
    -\delta_x\delta_xe_y^n\Big)+\delta_x\xi_z^{n+\frac12}\Big|_{i,j+\frac12,k}.
\en
%\normalsize
%whereafter $F|_{\al,\beta,\g}$ means each term of the equation $F$ has the subscripts $\al,\beta,\g.$\\
{\sf Stage 2 of the discrete form of the Maxwell equations:\hfill}
%\small
\be\label{Max2.1}
&&\frac{\delta_xe_x^{n+1}-\delta_xe_x^{n*}}{\dt/2}=\frac{1}{\vep}\Big(\delta_y\delta_xh_z^{n*}
     -\delta_z\delta_xh_y^{n+1}\Big)+\delta_x\beta_x^{n+\frac{1}{2}}\Big|_{i,j,k},\\ \label{Max2.2}
&&\frac{\delta_xe_y^{n+1}-\delta_xe_y^{n*}}{\dt/2}=\frac{1}{\vep}\Big(\delta_z\delta_xh_x^{n*}
  -\delta_x\delta_xh_z^{n+1}\Big)+\delta_x\beta_y^{n+\frac12}\Big|_{i+\frac12,j+\frac12,k},\\ \label{Max2.3}
&&\frac{\delta_xe_z^{n+1}-\delta_xe_z^{n*}}{\dt/2}=\frac{1}{\vep}\Big(\delta_x\delta_xh_y^{n*}
  -\delta_y\delta_xh_x^{n+1}\Big)+\delta_x\beta_z^{n+\frac12}\Big|_{i+\frac12,j,k+\frac12},\\ \label{Max2.4}
&&\frac{\delta_xh_x^{n+1}-\delta_xh_x^{n*}}{\dt/2}=\frac{1}{\mu}\Big(\delta_z\delta_xe_y^{n*}
  -\delta_y\delta_xe_z^{n+1}\Big)+\delta_x\xi_x^{n+\frac12}\Big|_{i+\frac12,j+\frac12,k+\frac12},\\
  \label{Max2.5}
&&\frac{\delta_xh_y^{n+1}-\delta_xh_y^{n*}}{\dt/2}=\frac{1}{\mu}\Big(\delta_x\delta_xe_z^{n*}
   -\delta_z\delta_xe_x^{n+1}\Big)+\delta_x\xi_y^{n+\frac12}\Big|_{i,j,k+\frac12},\\ \label{Max2.6}
&&\frac{\delta_xh_z^{n+1}-\delta_xh_z^{n*}}{\dt/2}=\frac{1}{\mu}\Big(\delta_y\delta_xe_x^{n*}
   -\delta_x\delta_xe_y^{n+1}\Big)+\delta_x\xi_z^{n+\frac12}\Big|_{i,j+\frac12,k}.
\en
%\normalsize
Here, in Stage 1 and Stage 2 of the discrete form of the Maxwell equations
$\delta_x\beta_w^{n+\frac{1}{2}}$ and $\delta_x\xi_w^{n+\frac{1}{2}}$ $(w=x,y,z)$
are the truncation error terms of the equivalent form of the $\delta_x$-ADI-FDTD scheme (see below).
For example,
%\small
\ben
&&\delta_x\beta_x^{n+\frac12}= \delta_x\delta_te_x^{n+\frac12}
   -\frac{1}{\vep}\delta_x\bar{\delta_t}\Big[\delta_yh_z^{n+\frac12}-\delta_zh_y ^{n+\frac12}\Big]
   -\frac{(\dt)^2}{4\mu\vep}\delta_x\delta_y\delta_x\delta_te_y^{n+\frac12}\Big|_{i,j,k},\\
&&\delta_x\xi_x^{n+\frac12}=\delta_x\delta_th_x^{n+\frac12}
   -\frac{1}{\mu}\delta_x\bar{\delta_t}\Big[\delta_ze_y^{n+\frac12}-\delta_ye_z^{n+\frac12}\Big]
   -\frac{(\dt)^2}{4\mu\vep}\delta_x\delta_z\delta_x\delta_th_z^{n+\frac12}\Big|_{\bar{i},\bar{j},\bar{k}}.
\enn
%\normalsize
The other terms can be obtained similarly. Then the equivalent form of the $\delta_x$-ADI-FDTD scheme
can be easily derived by combining all the equations in the $\delta_x$-ADI-FDTD scheme and
canceling the field values at the intermediate time levels $t^{n+\frac12}$ as follows:
%\small
\ben\no%\label{equs1}
\delta_x\delta_tE_x^{n+\frac12}=\frac{1}{\vep}\delta_x\bar{\delta_t}\Big[\delta_yH_z^{n+\frac12}
   -\delta_zH_y ^{n+\frac12}\Big]
   +\frac{(\dt)^2}{4\mu\vep}\delta_x\delta_y\delta_x\delta_tE_y^{n+\frac12}\Big|_{i,j,k},\\ \no
\delta_x\delta_tE_y^{n+\frac12}=\frac{1}{\vep}\delta_x\bar{\delta_t}\Big[\delta_zH_x^{n+\frac12}
   -\delta_xH_z^{n+\frac12}\Big]
   +\frac{(\dt)^2}{4\mu\vep}\delta_x\delta_z\delta_y\delta_tE_z^{n+\frac12}\Big|_{\bar{i},\bar{j},k},\\ \no
\delta_x\delta_tE_z^{n+\frac12}=\frac{1}{\vep}\delta_x\bar{\delta_t}\Big[\delta_xH_y^{n+\frac12}
   -\delta_yH_x^{n+\frac12}\Big]
   +\frac{(\dt)^2}{4\mu\vep}\delta_x\delta_x\delta_z\delta_tE_x^{n+\frac12}\Big|_{\bar{i},j,\bar{k}};\\
\label{equis}
\delta_x\delta_tH_x^{n+\frac12}=\frac{1}{\mu}\delta_x\bar{\delta_t}\Big[\delta_zE_y^{n+\frac12}
   -\delta_yE_z^{n+\frac12}\Big]+\frac{(\dt)^2}{4\mu\vep}
   \delta_x\delta_z\delta_x\delta_tH_z^{n+\frac12}\Big|_{\bar{i},\bar{j},\bar{k}},\\ \no
\delta_x\delta_tH_y^{n+\frac12}=\frac{1}{\mu}\delta_x\bar{\delta_t}\Big[\delta_xE_z^{n+\frac12}
   -\delta_zE_x^{n+\frac12}\Big]+\frac{(\dt)^2}{4\mu\vep}
   \delta_x\delta_x\delta_y\delta_tH_x^{n+\frac12}\Big|_{i,j,\bar{k}},\\ \no
\delta_x\delta_tH_z^{n+\frac12}=\frac{1}{\mu}\delta_x\bar{\delta_t}\Big[\delta_yE_x^{n+\frac12}
  -\delta_xE_y^{n+\frac12}\Big]+\frac{(\dt)^2}{4\mu\vep}
  \delta_x\delta_y\delta_z\delta_tH_y^{n+\frac12}\Big|_{i,\bar{j},k}.
\enn
%\normalsize

Removing the operator $\delta_x$ from all the terms in the above equations leads to
the equivalent form of the ADI-FDTD scheme which is omitted here for shortness.

From the above equivalent scheme and the Taylor expansion it can be seen that the
truncation terms in (\ref{Max1.1})-(\ref{Max2.6}) are bounded if the exact solution
of the Maxwell equations is suitably smooth, that is,
%\small
\be\label{trunc1.1a}
&&\qquad|\delta_x\beta^{\bar{n}}_{x_{i,j,k}}|+|\delta_x\beta^{\bar{n}}_{y_{\bar{i},\bar{j},k}}|
  +|\delta_x\beta^{\bar{n}}_{z_{\bar{i},j,\bar{k}}}|
  \le C_{eh}\{(\dt)^2+(\dx)^2+(\dy)^2+(\dz)^2\},\\ \label{trunc1.1b}
&&\qquad|\delta_x\xi^{\bar{n}}_{x_{\bar{i},\bar{j},\bar{k}}}|
  +|\delta_x\xi^{\bar{n}}_{y_{i,j,\bar{k}}}|+|\delta_x\xi^{\bar{n}}_{z_{i,\bar{j},k}}|
\le C_{eh}\{(\dt)^2+(\dx)^2+(\dy)^2+(\dz)^2\},
\en
%\normalsize
where $\bar{l}=l+\frac12$ for $l=i,j,k,n$ and $C_{eh}$ is a positive constant depending
on $\vep,\mu$ and the upper bounds of the derivatives of ${\bf e},{\bf h}$.

For $n\geq 0$, $w=x,y,z$ and for $\alpha=i,i+1/2$, $\beta=j,j+1/2$, $\gamma=k,k+1/2$
%let $E_w(t^n,x_\al,y_\beta,z_\g),$ $H_w(t^n,x_\al,y_\beta,z_\g)$ be the exact solution
%to the Maxwell equations (\ref{mw1})-(\ref{mw2}) evaluated at $(t^n,x_\al,y_\beta,z_\g)$.
let
\ben
&&\delta_x\mE^n_{w_{\alpha,\beta,\gamma}}
  =\delta_xe^n_{w_{\al,\beta,\g}}-\delta_xE^n_{w_{\alpha,\beta,\gamma}},\qquad
\delta_x\mH^n_{w_{\alpha,\beta,\gamma}}=\delta_xh^n_{w_{\al,\beta,\g}}
  -\delta_xH^n_{w_{\alpha,\beta,\gamma}},\\
&&\delta_x\mE^{n+\frac{1}{2}}_{w_{\alpha,\beta,\gamma}}
  =\delta_xe^{n*}_{w_{\al,\beta,\g}}-\delta_xE^n_{w_{\alpha,\beta,\gamma}},\qquad
\delta_x\mH^{n+\frac{1}{2}}_{w_{\alpha,\beta,\gamma}}
  =\delta_xh^{n*}_{w_{\al,\beta,\g}}-\delta_xH^n_{w_{\alpha,\beta,\gamma}},
\enn
where $\ds\me^n=(\mE^n_x,\mE^n_y,\mE^n_z)$ and $\ds\mh^n=(\mH^n_x,\mH^n_y,\mH^n_z).$
Then subtracting the $\delta_x$-ADI-FDTD scheme (i.e. the ADI-FDTD scheme applied by $\delta_x$)
%(see its first stage (\ref{t3.12}))
from the discrete form of the Maxwell's equations (see (\ref{Max1.1})-(\ref{Max2.6})) leads
to the following system of error equations:\\
%\small
{\sf Error-$\delta_x$-Stage 1:\hfill}
\ben\no
&&\delta_x\mE_x^{n+\frac12}-\frac{\dt}{2\vep}\delta_x\delta_y\mH_z^{n+\frac12}
  =\delta_x\mE_x^n-\frac{\dt}{2\vep}\delta_x\delta_z\mH_y^n
   +\frac{\dt}{2}{\delta_x\beta_x^{n+\frac{1}{2}}}\Big|_{i,j,k},\\ \no
&&\delta_x\mE_y^{n+\frac12}-\frac{\dt}{2\vep}\delta_x\delta_z\mH_x^{n+\frac12}
  =\delta_x\mE_y^n-\frac{\dt}{2\vep}\delta_x\delta_x\mH_z^n
   +\frac{\dt}{2}{\delta_x\beta_y^{n+\frac{1}{2}}}\Big|_{i+\frac12,j+\frac12,k},\\ \no
&&\delta_x\mE_z^{n+\frac12}-\frac{\dt}{2\vep}\delta_x\delta_x\mH_y^{n+\frac12}
  =\delta_x\mE_z^n-\frac{\dt}{2\vep}\delta_x\delta_y\mH_x^n
   +\frac{\dt}{2}{\delta_x\beta_z^{n+\frac{1}{2}}}\Big|_{i+\frac12,j,k+\frac12},\\ \label{Er4.1}
&&\delta_x\mH_x^{n+\frac12}-\frac{\dt}{2\mu}\delta_x\delta_z\mE_y^{n+\frac12}
  =\delta_x\mH_x^n-\frac{\dt}{2\mu}\delta_x\delta_y\mE_z^n
   +\frac{\dt}{2}{\delta_x\xi_x^{n+\frac{1}{2}}}\Big|_{i+\frac12,j+\frac12,k+\frac12},\\ \no
&&\delta_x\mH_y^{n+\frac12}-\frac{\dt}{2\mu}\delta_x\delta_x\mE_z^{n+\frac12}
  =\delta_x\mH_y^n-\frac{\dt}{2\mu}\delta_x\delta_z\mE_x^n
   +\frac{\dt}{2}{\delta_x\xi_y^{n+\frac{1}{2}}}\Big|_{i,j,k+\frac{1}{2}},\\ \no
&&\delta_x\mH_z^{n+\frac12}-\frac{\dt}{2\mu}\delta_x\delta_y\mE_x^{n+\frac12}
  =\delta_x\mH_z^n-\frac{\dt}{2\mu}\delta_x\delta_x\mE_y^n
   +\frac{\dt}{2}{\delta_x\xi_z^{n+\frac{1}{2}}}\Big|_{i,j+\frac12,k}.
\enn
%\normalsize
{\sf Error-$\delta_x$-Stage 2} can be obtained similarly. For example, the first error equation
is given by
%\small
\ben\label{Er4.2}
\delta_x\mE_x^{n+1}+\frac{\dt}{2\vep}\delta_x\delta_z\mH_y^{n+1}
   =\delta_x\mE_x^{n+\frac12}+\frac{\dt}{2\vep}\delta_x\delta_y\mH_z^{n+\frac12}
    +\frac{\dt}{2}{\delta_x\beta_x^{n+\frac{1}{2}}}\Big|_{i,j,k}.
\enn
%\normalsize
We also need the following system of error equations for the ADI-FDTD scheme, which is obtained
by repeating the argument above and letting $\delta_x=I$ (the identity operator):
%\small
{\sf Error-Stage 1:\hfill}
\ben\no
&&\vep\mE^{n+\frac12}_x-\frac{\dt}{2}\delta_y\mH^{n+\frac12}_z
   =\vep\mE^n_x-\frac{\dt}{2}\delta_z\mH^n_y
    +\frac{\vep\dt}{2}\beta_x^{n+\frac{1}{2}}\Big|_{i+\frac12,j,k},\\ \no
&&\vep\mE_y^{n+\frac{1}{2}}-\frac{\dt}{2}\delta_z\mH_x^{n+\frac{1}{2}}
   =\vep{\mathcal{E}}_y^n-\frac{\dt}{2}\delta_x{\mathcal{H}}_z^n
    +\frac{\varepsilon\dt}{2}\beta_y^{n+\frac{1}{2}}\Big|_{i,j+\frac{1}{2},k},\\ \no
&&\vep{\mathcal{E}}_z^{n+\frac{1}{2}}-\frac{\dt}{2}\delta_x{\mathcal{H}}_y^{n+\frac{1}{2}}
   =\vep{\mathcal{E}}_z^n-\frac{\dt}{2}\delta_y{\mathcal{H}}_x^n
     +\frac{\varepsilon\Delta t}{2}\beta_z^{n+\frac{1}{2}}\Big|_{i,j,k+\frac{1}{2}},\\ \label{4.1}
&&\mu{\mathcal{H}}_x^{n+\frac{1}{2}}-\frac{\dt}{2}\delta_z{\mathcal{E}}_y^{n+\frac{1}{2}}
   =\mu{\mathcal{H}}_x^n -\frac{\dt}{2}\delta_y{\mathcal{E}}_z^n
    +\frac{\mu\dt}{2}\xi_x^{n+\frac{1}{2}}\Big|_{i,j+\frac{1}{2},k+\frac{1}{2}},\\ \no
&&\mu{\mathcal{H}}_y^{n+\frac{1}{2}}-\frac{\dt}{2}\delta_x{\mathcal{E}}_z^{n+\frac{1}{2}}
   =\mu{\mathcal{H}}_y^n -\frac{\dt}{2}\delta_z{\mathcal{E}}_x^n
    +\frac{\mu\dt}{2}\xi_y^{n+\frac{1}{2}}\Big|_{i+\frac{1}{2},j,k+\frac{1}{2}},\\ \no
&&\mu{\mathcal{H}}_z^{n+\frac{1}{2}}-\frac{\dt}{2}\delta_y{\mathcal{E}}_x^{n+\frac{1}{2}}
   =\mu{\mathcal{H}}_z^n-\frac{\dt}{2}\delta_x{\mathcal{E}}_y^n
     +\frac{\mu\dt}{2}\xi_z^{n+\frac{1}{2}}\Big|_{i+\frac{1}{2},j+\frac{1}{2},k}.
\enn
%\normalsize
{\sf Error-Stage 2} can be obtained similarly. For example, the first equation
in this stage is given by
%\small
\ben\label{4.2}
\mE^{n+1}_x+\frac{\dt}{2\vep}\delta_z\mH^{n+1}_y=\mE^{n+\frac12}_x
   +\frac{\dt}{2\vep}\delta_y\mH^{n+\frac12}_z
   +\frac{\dt}{2}\beta^{n+\frac{1}{2}}_x\Big|_{i+\frac12,j,k}.
\enn
%\normalsize
In the above error equations, $\mE^m_{w_{\alpha,\beta,\gamma}}$, $\mH^m_{w_{\alpha,\beta,\gamma}}$
$\beta^{\bar{n}}_w$ and $\xi^{\bar{n}}_w$ with $w=x,y,z,m=n,\bar{n}$, can be regarded as
those in {\sf Error-$\delta_x$-Stages 1 and 2} with $\delta_x$ replaced with the identity operator $I$.
For example,
%, where
% ${\bf e}^{n*}=(e_x^{n*},e_y^{n*},e_z^{n*})$ and ${\bf h}^{n*}=(h_x^{n*},h_y^{n*},h_z^{n*})$,
% &&e_y^{n*}=\frac{1}{2}(e_y^{n+1}+e_y^n)+\frac{\dt}{4\vep}
%   \delta_x(h_z^{n+1}-h_z^n)|_{i,j+\frac{1}{2},k},\\
% &&e_z^{n*}=\frac{1}{2}(e_z^{n+1}+e_z^n)+\frac{\dt}{4\vep}
%   \delta_y(h_x^{n+1}-h_x^n)|_{i,j,k+\frac{1}{2}},\\
% &&h_x^{n*}=\frac{1}{2}(h_x^{n+1}+h_x^n)+\frac{\dt}{4\mu}
%   \delta_y(e_z^{n+1}-e_z^n)|_{i,j+\frac{1}{2},k+\frac{1}{2}},\\
% &&h_y^{n*}=\frac{1}{2}(h_y^{n+1}+h_y^n)+\frac{\dt}{4\mu}
%   \delta_z(e_x^{n+1}-e_x^n)|_{i+\frac{1}{2},j,k+\frac{1}{2}},\\
%\small
\ben
&&\mE_x^{n+\frac{1}{2}}={ e}_x^{n*}-E^n_x|_{i+\frac{1}{2},j,k},\qquad
\mH^{n+\frac{1}{2}}_z={h}_z^{n*}-H^n_z|_{i+\frac{1}{2},j+\frac{1}{2},k},\\
&& e_x^{n*}=\frac12(e_x^{n+1}+e_x^n)+\frac{\dt}{4\vep}\delta_z(h_y^{n+1}-h_y^n)|_{i+\frac12,j,k},\\
&& h_z^{n*}=\frac12(h_z^{n+1}+h_z^n)+\frac{\dt}{4\mu}\delta_x(e_y^{n+1}-e_y^n)|_{i+\frac12,j+\frac12,k}.
\enn
%\normalsize
$e_y^{n*}$, $e_z^{n*}$, $h_x^{n*}$ and $h_y^{n*}$ can be similarly defined.
Note that $\beta^{n+\frac12}_w$ and $\xi^{n+\frac12}_w$ with $w=x,y,z$ are the truncation error terms
for each equation of the equivalent form of the ADI-FDTD scheme
(i.e. the equivalent form of the $\delta_x$-ADI-FDTD scheme above with $\delta_x=I$). For example,
%\small
\ben
&&\beta_x^{n+\frac12}=\delta_te_x^{n+\frac12}-\frac{1}{\vep}\bar{\delta_t}
  \Big[\delta_yh_z^{n+\frac12}-\delta_zh_y ^{n+\frac12}\Big]
  -\frac{(\dt)^2}{4\mu\vep}\delta_y\delta_x\delta_te_y^{n+\frac12}\Big|_{i+\frac12,j,k},\\
&&\xi_x^{n+\frac12}=\delta_th_x^{n+\frac12}-\frac{1}{\mu}\bar{\delta_t}
  \Big[\delta_ze_y^{n+\frac12}-\delta_ye_z^{n+\frac12}\Big]
  -\frac{(\dt)^2}{4\mu\vep}\delta_z\delta_x\delta_th_z^{n+\frac12}\Big|_{i,j+\frac12,k+\frac12}.
\enn
%\normalsize

If the solution of the Maxwell equations is suitably smooth
%(e.g. $\E\in C^2((0,T];C^3(\bar\Om))$, $\Hb\in C^2((0,T];C^3(\bar\Om))$)
then the following estimates hold
%\small
\be\label{trunc1.2a}
&&\quad|\beta^{n+\frac12}_{x_{\bar{i},j,k}}|+|\beta^{n+\frac12}_{y_{i,\bar{j},k}}|
   +|\beta^{n+\frac12}_{z_{i,j,\bar{k}}}|\le C_{eh}[(\dt)^2+(\dx)^2+(\dy)^2+(\dz)^2],\\ \label{trunc1.2b}
&&\quad|\xi^{n+\frac{1}{2}}_{x_{i,\bar{j},\bar{k}}}|+|\xi^{n+\frac12}_{y_{\bar{i},j,\bar{k}}}|
   +|\xi^{n+\frac{1}{2}}_{z_{\bar{i},\bar{j},k}}|\le C_{eh}[(\dt)^2+(\dx)^2+(\dy)^2+(\dz)^2],
\en
%\normalsize
where $C_{eh}$ is a positive constant depending on $\vep,\mu$ and
the upper bounds of the derivatives of ${\bf e},{\bf h}$.

we now estimate (\ref{t4.1a}) by making use of these error equations.
Multiplying both sides of the first and second three equations in {\sf Error-$\delta_x$-Stage 1}
%(\ref{4.3})
by $\sqrt{\vep}$ and $\sqrt{\mu}$, respectively, and taking the square of the resulting
equations, we will obtain six equations which are similar to the six equations
(\ref{3.1a})-(\ref{3.1f}) in the proof of Theorem \ref{t3.1}, but each equation has
one extra term $f_i^n$ ($i=1,...,6$) which related to the local truncation error term.
For example, the first equation is given by
%\small
\ben\label{4.5a}
&&\quad\vep(\delta_x\mE_x^{n+\frac12})^2
      +\frac{(\dt)^2}{4\mu\vep}\mu(\delta_x\delta_y\mH_z^{n+\frac12})^2
      -\dt\delta_x\mE_x^{n+\frac12}\cdot\delta_x\delta_y\mH_z^{n+\frac12}\\ \no
&&\qquad\qquad=\vep(\delta_x\mE_x^n)^2+\frac{(\dt)^2}{4\mu\vep}\mu(\delta_x\delta_z\mH_y^n)^2
     -\dt\delta_x\mE_x^n\cdot\delta_x\delta_z\mH_y^n+f_1^n\Big|_{i,j,k},
\enn
%\normalsize
where
%$f_i^n$ $(i=1,...,6)$ can be defined obviously; e.g.
%\small
\ben
f_1^n=\frac{\vep}4(\dt)^2(\delta_x\beta_x^{n+\frac{1}{2}})^2+\vep\dt(\delta_x\mE_x^n
        -\frac{\dt}{2\vep}\delta_x\delta_z\mH^n_y)\delta_x\beta_x^{n+\frac{1}{2}}.
\enn
%\normalsize
The other terms $f_i^n$ $(i=2,...,6)$ can be defined similarly and obviously.
%f_2^n&=&\frac{\vep}{4}(\dt)^2(\delta_x\beta_y^{n+\frac{1}{2}})^2+\vep\dt(\delta_x\mE_y^n
%        -\frac{\dt}{2\vep}\delta_x\delta_x\mH^n_z)D_x\beta_y^{n+\frac{1}{2}},\\
%f_3^n&=&\frac{\vep}{4}(\dt)^2(D_x\beta_z^n)^2+\vep\dt(\delta_x\mE_z^n
%        -\frac{\dt}{2\vep}\delta_x\delta_y\mH^n_x)D_x\beta_z^n,\\
%f_4^n&=&\frac{\mu}{4}(\dt)^2(D_x\xi_x^n)^2+\mu\dt(\delta_x\mH_x^n
%      -\frac{\dt}{2\mu}\delta_x\delta_y\mE^n_z)D_x\xi_x^n,
%f_5^n&=&\frac{\mu}{4}(\dt)^2(D_x\xi_y^n)^2+\mu\dt(\delta_x\mH_y^n
%      -\frac{\dt}{2\mu}\delta_x\delta_z\mE^n_x)D_x\xi_y^n,\\
%f_6^n&=&\frac{\mu}{4}(\dt)^2(D_x\xi_z^n)^2+\mu\dt(\delta_x\mH_z^n
%      -\frac{\dt}{2\mu}\delta_x\delta_x\mE^n_y)D_x\xi_z^n.
%\enn

Similarly, we can obtain six equations from {\sf Error-$\delta_x$-Stage 2}.
For example, the first equation in this stage can be obtained from (\ref{4.2}):
%\small
\ben
&&\vep(\delta_x\mE_x^{n+1})^2+\frac{(\dt)^2}{4\mu\vep}\mu(\delta_x\delta_z\mH_y^{n+1})^2
     +\dt\delta_x\mE_x^{n+1}\cdot\delta_x\delta_z\mH_y^{n+1}+\bar{f}_1^{n+1}\\ \label{4.6}
&&\qquad=\vep(\delta_x\mE_x^{n+\frac12})^2
     +\frac{(\dt)^2}{4\mu\vep}\mu(\delta_x\delta_y\mH_z^{n+\frac12})^2
     +\dt\delta_x\mE_x^{n+\frac12}\cdot\delta_x\delta_y\mH_z^{n+\frac12}\Big|_{i+\frac12,j,k},
\enn
%\normalsize
where
%\small
\ben
\bar{f}_1^{n+1}=\frac{\vep}{4}(\dt)^2(\delta_x\beta_x^{n+\frac12})^2
 -\vep\dt(\delta_x\mE_x^{n+1}+\frac{\dt}{2\vep}\delta_x\delta_z\mH^{n+1}_y)\delta_x\beta_x^{n+\frac12}.
\enn
%\normalsize
The other terms $\bar{f}_i^{n+1}\;(i=2,\cdots,6)$ corresponding to the other five equations
are similar to $\bar{f}_1^{n+1}.$
%
%&&\bar{f}_2^{n+1}=\frac{\vep}{4}(\dt)^2(\frac{\pa\beta_y^{n+1}}{\pa x})^2
%-\vep\dt(\delta_x\mE_y^{n+1}+\frac{\dt}{2\vep}\delta_x\delta_x\mH^{n+1}_z)\frac{\pa\beta_y^{n+1}}{\pa x},\\
%&&\bar{f}_3^{n+1}=\frac{\vep}{4}(\dt)^2(\frac{\pa\beta_z^{n+1}}{\pa x})^2
%-\vep\dt(\delta_x\mE_z^{n+1}+\frac{\dt}{2\vep}\delta_x\delta_y\mH^{n+1}_x)\frac{\pa\beta_z^{n+1}}{\pa x},\\
%&&\bar{f}_4^{n+1}=\frac{\mu}{4}(\dt)^2(\frac{\pa \xi_x^{n+1}}{\pa x})^2
%-\mu\dt(\delta_x\mH_x^{n+1}+\frac{\dt}{2\mu}\delta_x\delta_y\mE^{n+1}_z)\frac{\pa\xi_x^{n+1}}{\pa x},\\
%&&\bar{f}_5^{n+1}=\frac{\mu}{4}(\dt)^2(\frac{\pa \xi_y^{n+1}}{\pa
%x})^2-\mu\dt(\delta_x\mH_y^{n+1}+\frac{\dt}{2\mu}\delta_x\delta_z\mE^{n+1}_x)\frac{\pa
%\xi_y^{n+1}}{\pa x},\\
%&&\bar{f}_6^{n+1}=\frac{\mu}{4}(\dt)^2(\frac{\pa \xi_z^{n+1}}{\pa
%x})^2-\mu\dt(\delta_x\mH_z^{n+1}+\frac{\dt}{2\mu}\delta_x\delta_x\mE^{n+1}_y)\frac{\pa
%\xi_z^{n+1}}{\pa x}.
%\enn
%

Using exactly the same argument as in the proof of Theorem \ref{t3.1} we arrive at
%\small
\be\label{4.7}
&&\quad\|\delta_x\mE^{n+1}\|_E^2+\|\delta_x\mH^{n+1}\|^2_H
      +\frac{(\dt)^2}{4\mu\vep}(\|\delta_x\delta_1^h\mE^{n+1}\|_H^2
      +\|\delta_x\delta_2^h\mH^{n+1}\|^2_E)\\ \no
&&\qquad\qquad+\|\mE^{n+1}\|^2_I+\|\mH^{n+1}\|^2_I
       +\frac{(\dt)^2}{4\mu\vep}\left(\|\delta_2^h\mH^{n+1}\|^2_I
       +\|\delta_1^h\mE^{n+1}\|^2_I\right)\\ \no
&&\qquad\;=\|\delta_x\mE^n\|_E^2+\|\delta_x\mH^n\|^2_H
      +\frac{(\dt)^2}{4\mu\vep}(\|\delta_x\delta_1^h\mE^n\|^2_H
      +\|\delta_x\delta_2^h\mH^n\|^2_E)\\ \no
&&\qquad\qquad+\|\mE^{n}\|^2_I+\|\mH^{n}\|^2_I
      +\frac{(\dt)^2}{4\mu\vep}\left(\|\delta_2^h\mH^n\|^2_I
      +\|\delta_1^h\mE^n\|^2_I\right)\\ \no
&&\qquad\qquad+F^{n,n+1}\Delta v+G^{n,n+1}\dy\dz,
\en
%\normalsize
where $\Delta v=\dx\dy\dz$ and for $n\geq 0$ we have
%\small
\ben
&&F^{n,n+1}=\sum\limits_{i=1}^{I-1}\sum\limits_{j=1}^{J-1}\sum\limits_{k=1}^{K-1}
  \big[f_1^n-\bar{f}_1^{n+1}\big]_{i,j,k}
  +\sum\limits_{i=1}^{I-2}\sum\limits_{j=0}^{J-1}\sum\limits_{k=1}^{K-1}
  \big[f_2^n-\bar{f}_2^{n+1}\big]_{i+\frac12,j+\frac12,k}\\
&&\;\;+\sum\limits_{i=1}^{I-2}\sum\limits_{j=1}^{J-1}\sum\limits_{k=0}^{K-1}
  \big[f_3^n-\bar{f}_3^{n+1}\big]_{i+\frac12,j,k+\frac12}
  +\sum\limits_{i=1}^{I-2}\sum\limits_{j=0}^{J-1}\sum\limits_{k=0}^{K-1}
  \big[f_4^n-\bar{f}_4^{n+1}\big]_{i+\frac12,j+\frac12,k+\frac12}\\
&&\;\;+\sum\limits_{i=1}^{I-1}\sum\limits_{j=1}^{J-1}\sum\limits_{k=0}^{K-1}
  \big[f_5^n-\bar{f}_5^{n+1}\big]_{i,j,k+\frac{1}{2}}
  +\sum\limits_{i=1}^{I-1}\sum\limits_{j=0}^{J-1}\sum\limits_{k=1}^{K-1}
  \big[f_6^n-\bar{f}_6^{n+1}\big]_{i,j+\frac{1}{2},k}\\
&&\quad\quad:=I^n_1+I^n_2+I^n_3+I^n_4+I^n_5+I^n_6,
\enn
\ben
G^{n,n+1}&=&(G_1^{n,n+1}+G_2^{n,n+1}+G_3^{n,n+1})/\dx,\\
G_1^{n,n+1}&=&\sum\limits_{i^{\prime}=1,I-1}\sum\limits_{j=1}^{J-1}\sum\limits_{k=0}^{K-1}
   \Big[\vep\dt\big(\mE^{n+1}_z+\mE^n_z)+\frac{(\dt)^2}{2}\delta_y\big(\mH^{n+1}_x
   -\mH^n_x\big)\Big]\beta^{n+\frac{1}{2}}_z\Big|_{i',j,\bar{k}},\\
G_2^{n,n+1}&=&\sum\limits_{i'=1,I-1}\sum\limits_{j=0}^{J-1}\sum\limits_{k=1}^{K-1}
  \Big[\vep\dt\big(\mE^{n+1}_y+\mE^n_y\big)+\frac{(\dt)^2}{2}\delta_x\big(\mH^{n+1}_z
  -\mH^n_z\big)\Big]\beta^{n+\frac{1}{2}}_y\Big|_{i',\bar{j},k},\\
G_3^{n,n+1}&=&\sum\limits_{i^{\prime}=1,I-1}\sum\limits_{j=0}^{J-1}\sum\limits_{k=0}^{K-1}
  \Big[\mu\dt\big(\mH^{n+1}_x+\mH^n_x\big)+\frac{(\dt)^2}{2}\delta_y\big(\mE^{n+1}_z
  +\mE^n_z\big)\Big]\xi^{n+\frac{1}{2}}_x\Big|_{i',\bar{j},\bar{k}}.
\enn
%\normalsize
We now estimate each term in $F^{n,n+1}$ and $G^{n,n+1}$.
Noting that by the initial condition $\mH^0_y=\mE^0_x=0$, we have
%\small
\be\no
&&\qquad\sum\limits_{l=0}^n I^l_1=\sum\limits_{l=0}^n\sum\limits_{i=1}^{I-1}
   \sum\limits_{j=1}^{J-1}\sum\limits_{k=1}^{K-1}
  \vep\dt\Big[\delta_x(\mE_x^{l+1}+\mE_x^{l})
  +\frac{\dt}{2\vep}\delta_x\delta_z(\mH_y^{l+1}
  -\mH_y^{l})\Big]\delta_x\beta_x^{l+\frac12}\big|_{i,j,k}\\ \no
&&\qquad\qquad\leq\sum\limits_{i=1}^{I-1}\sum\limits_{j=1}^{J-1}\sum\limits_{k=1}^{K-1}
  \Big[\frac12\frac{(\dt)^2}{4\mu\vep}\mu(\delta_x\delta_z\mH^{n+1}_y)^2
  +\frac{\vep}2(\delta_x\mE^{n+1}_x)^2\\ \no
&&\qquad\qquad\;+\dt\sum\limits_{l=0}^n\{\vep(\delta_x\mE^l_x)^2
  +\frac{(\dt)^2}{4\mu\vep}\mu(\delta_x\delta_z\mH_y^l)^2\}\Big]_{i,j,k}\\ \label{4.8}
&&\qquad\qquad\;+C\{(\dt)^4+(\dx)^4+(\dy)^4+(\dz)^4\}.
\en
%\normalsize
%where we have used the inequality
%\ben\label{4.9}
%\frac{(\dt)^2}{2}ab\leq\frac12\frac{(\dt)^2}{4\mu\vep}\mu a^2+\frac{\vep(\dt)^2}2 b^2
%\enn
%with $a=\delta_x\delta_z\mH^n_y$ and $b=\delta_x\beta_x^{n+\frac{1}{2}}$.
Similar estimates can be obtained for $I_i^n$ with $i=2,...,6$.
Add these estimates together and use the definition of the discrete norms we get
%\small
\be\no
\sum\limits_{l=0}^nF^{l,l+1}\Delta v&\le&\frac{1}{2}\frac{(\dt)^2}{4\mu\vep}
    \Big\{\|\delta_x\delta_1^h\mE^{n+1}\|_H^2+\|\delta_x\delta_2^h\mH^{n+1}\|^2_E\Big\}\\ \no
&&+\frac{1}{2}\big(\|\delta_x\mE^{n+1}\|_E^2+\|\delta_x\mH^{n+1}\|^2_H\big)\\ \no
&&+\dt\sum\limits_{l=0}^n\Big(\|\delta_x\mE^l\|_E^2+\|\delta_x\mH^l\|^2_H
  +\frac{(\dt)^2}{4\mu\vep}(\|\delta_x\delta_1^h\mE^l\|_H^2
   +\|\delta_x\delta_2^h\mH^l\|_E^2)\Big)\\ \label{4.9a}
&&   +C\{(\dt)^4+(\dx)^4+(\dy)^4+(\dz)^4\}.
\en
%\normalsize
Arguing similarly as in deriving (\ref{4.8}) we can derive that
%\small
\be\label{4.9b}
\sum\limits_{l=0}^{n}G_1^{l,l+1}&\le&\sum\limits_{i^\prime=1,I-1}
    \sum\limits_{j=1}^{J-1}\sum\limits_{k=0}^{K-1}
    \Big[\frac{1}{2}\big(\frac{(\dt)^2}{4\mu\vep}\mu(\delta_y\mH_x^{n+1})^2+\vep(\mE_z^{n+1})^2\big)\\ \no
&&+\dt\sum\limits_{l=0}^n\big(\vep(\mE_z^{l})^2
    +\frac{(\dt)^2}{4\mu\vep}(\delta_y\mH_x^{l})^2\big)\Big]_{i',j,k+\frac12}+C(\Delta)^4,\\ \label{4.9c}
\sum\limits_{l=0}^nG_2^{l,l+1}&\le&\sum\limits_{i'=1,I-1}\sum\limits_{j=0}^{J-1}\sum\limits_{k=1}^{K-1}
  \Big[\frac{1}{2}\big(\frac{(\dt)^2}{4\mu\vep}\mu(\delta_x\mH^{n+1}_z)^2+\vep(\mE^{n+1}_y)^2\big)\\ \no
&&+\dt\sum\limits_{l=0}^n\big((\mE^l_y)^2
  +\frac{(\dt)^2}{4\mu\vep}\mu(\delta_x\mH_z)^2\big)\Big]_{i',j+\frac12,k}+C(\Delta)^4;\\ \label{4.9d}
\sum\limits_{l=0}^{n}G_3^{l,l+1}&\le&\sum\limits_{i^\prime=1,I-1}
    \sum\limits_{j=1}^{J-1}\sum\limits_{k=0}^{K-1}
    \Big[\frac12\big(\frac{(\dt)^2}{4\mu\vep}\vep(\delta_y\mE_z^{n+1})^2+\mu(\mH_x^{n+1})^2\big)\\ \no
&&+\dt\sum\limits_{l=0}^n\big(\mu(\mH_x^l)^2
    +\frac{(\dt)^2}{4\mu\vep}\vep(\delta_y\mE_z^l)^2\big)\Big]_{i',j+\frac12,k+\frac12}+C(\Delta)^4,
\en
%\normalsize
where $(\Delta)^4=(\dt)^4+(\dx)^4+(\dy)^4+(\dz)^4$.
Combining (\ref{4.9b})-(\ref{4.9d}) and using the definition of the discrete norms we obtain that
%\small
\be\no
&&\sum\limits_{n=0}^nG^{n,n+1}\dy\dz\le\frac12\Big\{\frac{(\dt)^2}{4\mu\vep}
   \big(\|\delta_1^h\mE^{n+1}\|^2_I+\|\delta_2^h\mH^{n+1}\|^2_I\big)
   +\|\mE^{n+1}\|_I^2+\|\mH^{n+1}\|^2_I \Big\}\\ \label{4.9f}
&&\qquad\qquad+\dt\sum\limits_{l=0}^{n}\Big(\frac{(\dt)^2}{4\mu\vep}\big(\|\delta_1^h\mE^{l}\|^2_I
   +\|\delta_2^h\mH^{l}\|^2_I\big)+\|\mE^l\|^2_I+\|\mH^l\|^2_I\Big)+C(\Delta)^4.
\en
%\normalsize
%where
%$\delta_1^h\mE^m|_{i^\prime,\alpha,\beta}=(\delta_y{\mE_z}_{i^\prime,\bar{j},\bar{k}}^m,0,0)$,
%$\delta_2^h\mH^m|_{i^\prime,\alpha,\beta}
% =(0,\delta_x{\mH_z}_{i^\prime,\bar{j},k}^m,\delta_y{\mH_x}_{i^\prime,j,\bar{k}}^m)$,
%$\mE^m|_{i^\prime,\alpha,\beta}=(0,{\mE_y}_{i^\prime,\bar{j},k}^m,{\mE_z}_{i^\prime,j,\bar{k}}^m)$
%and $\mH^m|_{i^\prime,\alpha,\beta}=({\mH_x}_{i^\prime,\bar{j},\bar{k}}^m,0,0)$.

By summing up both sides of (\ref{4.7}) over the time levels $n$ and using (\ref{4.9a}) and (\ref{4.9f})
it follows that
%\small
\be\label{4.10}
&&\quad\|\delta_x\mE^{n+1}\|_E^2+\|\delta_x\mH^{n+1}\|^2_H
   +\frac{(\dt)^2}{4\mu\vep}\big(\|\delta_x\delta_1^h\mE^{n+1}\|_H^2
   +\|\delta_x\delta_2^h\mH^{n+1}\|^2_E\big)\\ \no
&&\qquad\qquad+\|\mE^{n+1}\|^2_I+\|\mH^{n+1}\|^2_I
   +\frac{(\dt)^2}{4\mu\vep}\Big[\|\delta_2^h\mH^{n+1}\|^2_I+\|\delta_1^h\mE^{n+1}\|^2_I\Big]\\ \no
&&\quad\le C\dt\sum\limits_{l=0}^{n}\Big[\|\delta_x\mE^l\|_E^2+\|\delta_x\mH^l\|^2_H
   +\|\mE^l\|^2_I+\|\mH^l\|^2_I\\ \no
&&\qquad\qquad+\frac{(\dt)^2}{4\mu\vep}\big(\|\delta_x\delta_1^h\mE^{l}\|_H^2
   +\|\delta_x\delta_2^h\mH^{l}\|^2_E+\|\delta_2^h\mH^{l}\|^2_I
   +\|\delta_1^h\mE^{l}\|^2_I\big)\Big]+C(\Delta)^4.
\en
%\normalsize
%where $(\Delta)^4=(\dt)^4+(\dx)^4+(\dy)^4+(\dz)^4$.
The discrete Gronwall's Lemma \ref{Gronwall} implies the estimate (\ref{t4.1a}) with $w=x$.
%\ben\label{4.11}
%&&\|\delta_x\mE^{n+1}\|_E^2+\|\delta_x\mH^{n+1}\|^2_H
%     +\frac{(\dt)^2}{4\mu\vep}\Big[\|\delta_x\delta_1^h\mE^{n+1}\|_H^2
%     +\|\delta_x\delta_2^h\mH^{n+1}\|^2_E\Big]\\
%&&\qquad\qquad\qquad+\|\mE^{n+1}\|^2_I+\|\mH^{n+1}\|^2_I
%     +\frac{(\dt)^2}{4\mu\vep}\Big[\|\delta_2^h\mH^{n+1}\|^2_I
%     +\|\delta_1^h\mE^{n+1}\|^2_I\Big]\\
%&&\qquad\qquad\le C\{(\dt)^4+(\dx)^4 +(\dy)^4+(\dz)^4\}.
%\enn
This completes the proof.
\end{proof}

If we apply the difference operators $\delta_t$ and $\delta_w\delta_t$ with $w=x,y,z$ to the equations
in the ADI-FDTD scheme and argue similarly as in the proof of Theorem \ref{t4.1},
then we are able to obtain the following result.

\begin{theorem}\label{t4.2}
Suppose the exact solution ${\bf e}$, ${\bf h}$ satisfies the same conditions as in Theorem $\ref{t4.1}.$
%Suppose $\E$ and $\Hb$ are fourth-order continuously differentiable on $\bar{\Om},$
%${\pa\E}/{\pa t}$ and ${\pa\Hb}/{\pa t}$ are second-order continuously differentiable on $\bar{\Om}$,
%${\pa^2\E}/{\pa t^2}$ and ${\pa^2\Hb}/{\pa t^2}$ are first-order continuously differentiable on $\bar{\Om}$
%and ${\pa^3\E}/{\pa t^3}$ and ${\pa^3\Hb}/{\pa t^3}$ are continuous on $\bar{\Omega}$ and
%assume that all the previously mentioned derivatives are continuous in time.
%and let $\E^n$ and $\Hb^n$ be the solution of the ADI-FDTD scheme.
Then for any $n\geq 0$ and any fixed $T>0$ there is a constant $C$ independent of $\dt,\dx,\dy,\dz$
such that for any $w=x,y,z$,
%\small
\be\no
&&\|\delta_w\delta_t\mE^{n+\frac{1}{2}}\|_E^2+\|\delta_w\delta_t\mH^{n+\frac{1}{2}}\|^2_H
  +\frac{(\dt)^2}{4\mu\vep}(\|\delta_1^h\delta_w\delta_t\mE^{n+\frac{1}{2}}\|_H^2
  +\|\delta_2^h\delta_w\delta_t\mH^{n+\frac{1}{2}}\|^2_E)\\ \no
&&\quad+\|\delta_t\mE^{n+\frac{1}{2}}\|^2_{L(w)}+\|\delta_t\mH^{n+\frac{1}{2}}\|^2_{L(w)}
   +\frac{(\dt)^2}{4\mu\vep}\Big[\|\delta_2^h\delta_t\mH^{n+\frac{1}{2}}\|^2_{L(w)}
   +\|\delta_1^h\delta_t\mE^{n+\frac{1}{2}}\|^2_{L(w)}\Big]\\ \label{t4.2a}
&&\qquad\qquad\quad\le C\{(\dt)^4+(\dx)^4+(\dy)^4+(\dz)^4\},
\en
where $L(x)=I$, $L(y)=J$ and $L(z)=K$.
\end{theorem}

\begin{proof}
We only prove (\ref{t4.2a}) for the case with $w=x.$ The other cases with $w=y,z$
can be proved similarly.

Repeating the argument in the proof of Theorem \ref{t4.1} with $n$ being replaced by $n-1/2$
and each function $U$ being replaced by $\delta_t U$, we are able to obtain that for any $n\ge1$,
%\small
\ben\no
&&\|\delta_x\delta_t\mE^{n+\frac12}\|_E^2+\|\delta_x\delta_t\mH^{n+\frac12}\|^2_H
   +\frac{(\dt)^2}{4\mu\vep}(\|\delta_x\delta_1^h\delta_t\mE^{n+\frac12}\|_H^2
   +\|\delta_x\delta_2^h\delta_t\mH^{n+\frac12}\|^2_E)\\ \no
&&\qquad\qquad+\|\delta_t\mE^{n+\frac12}\|^2_I+\|\delta_t\mH^{n+\frac12}\|^2_I
   +\frac{(\dt)^2}{4\mu\vep}\{\|\delta_2^h\delta_t\mH^{n+\frac12}\|^2_I
   +\|\delta_1^h\delta_t\mE^{n+\frac12}\|^2_I\\ \no
&&\qquad\le C\Big\{\|\delta_x\delta_t\mE^{\frac12}\|_E^2
   +\|\delta_x\delta_t\mH^{\frac12}\|^2_H
   +\frac{(\dt)^2}{4\mu\vep}(\|\delta_x\delta_1^h\delta_t\mE^{\frac12}\|_H^2
   +\|\delta_x\delta_2^h\delta_t\mH^{\frac12}\|^2_E)\\ \no
&&\qquad\qquad+\|\delta_t\mE^{\frac12}\|^2_I+\|\delta_t\mH^{\frac12}\|^2_I
   +\frac{(\dt)^2}{4\mu\vep}\{\|\delta_t\delta_2^h\mH^{\frac12}\|^2_I
   +\|\delta_1^h\delta_t\mE^{\frac12}\|^2_I\\ \label{4.14}
&&\qquad\qquad\qquad+(\dt)^4+(\dx)^4+(\dy)^4+(\dz)^4\Big\}.
\enn
%\normalsize
Since $\mE^0=\mH^0=0$, $\delta_t\mE^{1/2}=(1/\dt)\mE^1$ and
$\delta_t\mH^{1/2}=(1/\dt)\mH^1$, it is enough to prove
%\small
\be\label{4.13}
&&\|\delta_x\mE^1\|_E^2+\|\delta_x\mH^1\|^2_H
    +\frac{(\dt)^2}{4\mu\vep}\big(\|\delta_x\delta_1^h\mE^1\|_H^2
    +\|\delta_x\delta_2^h\mH^1\|^2_E\big)\\ \no
&&\qquad\qquad +\|\mE^1\|^2_I+\|\mH^1\|^2_I
    +\frac{(\dt)^2}{4\mu\vep}\Big(\|\delta_2^h\mH^1\|^2_I+\|\delta_1^h\mE^1\|^2_I\Big)\\ \no
&&\qquad\qquad\qquad\le C(\dt)^2\big[(\dt)^4+(\dx)^4+(\dy)^4+(\dz)^4\big].
\en
%\normalsize
Letting $n=0$ in (\ref{4.7}) and using the fact that $\mE^0=\mH^0=0$ we find that
%\small
\be\label{4.13a}
&&\|\delta_x\mE^1\|_E^2+\|\delta_x\mH^1\|^2_H
    +\frac{(\dt)^2}{4\mu\vep}\big(\|\delta_x\delta_1^h\mE^1\|_H^2
    +\|\delta_x\delta_2^h\mH^1\|^2_E\big)\\ \no
&&\qquad\qquad +\|\mE^1\|^2_I+\|\mH^1\|^2_I
    +\frac{(\dt)^2}{4\mu\vep}\Big(\|\delta_2^h\mH^1\|^2_I+\|\delta_1^h\mE^1\|^2_I\Big)\\ \no
&&\qquad\qquad\qquad\le F^{0,1}\Delta v+G^{0,1}\dy\dz,
\en
%\normalsize
where $\Delta v=\dx\dy\dz.$
We now estimate each term of $F^{0,1}$ and $G^{0,1}$ slightly differently
from the proof of Theorem \ref{t4.1} (cf. (\ref{4.8})-(\ref{4.9f}))
since we do not need to take the sum with respect to $n$ here.
For example, taking $n=0$ in (\ref{4.8}) we have, on noting that $\mE^0_x=0$, that
%\small
\ben
I^0_1&=&\sum\limits_{i=1}^{I-1}\sum\limits_{j=1}^{J-1}\sum\limits_{k=1}^{K-1}
        \Big[\vep\dt(\delta_x\mE_x^1
        +\frac{\dt}{2\vep}\delta_x\delta_z\mH_y^1)\delta_x\beta_x^{\frac12}\Big]_{i,j,k}\\ \no
&\le&\sum\limits_{i=1}^{I-1}\sum\limits_{j=1}^{J-1}\sum\limits_{k=1}^{K-1}
  \Big[\frac12\frac{(\dt)^2}{4\mu\vep}\mu(\delta_x\delta_z\mH^{1}_y)^2
  +\frac{\vep}2(\delta_x\mE^{1}_x)^2\Big]_{i,j,k}\\ \no
&&\qquad\qquad+C(\dt)^2\big[(\dt)^4+(\dx)^4 +(\dy)^4+(\dz)^4\big].
\enn
%\normalsize
The other five terms $I_i^0$, $i=2,...,6$, of $F^{0,1}$ and $G^{0,1}_i$, $i=1,2,3,$ can be
estimated similarly. Then by the definition of the discrete norms it is derived (cf. (\ref{4.9f})) that
%\small
\ben
&&F^{0,1}\Delta v+G^{0,1}\dy\dz\le\frac12\frac{(\dt)^2}{4\mu\vep}\Big\{\|\delta_x\delta_1^h\mE^1\|_H^2
    +\|\delta_x\delta_2^h\mH^1\|^2_E\\
&&\qquad+\|\delta_1^h\mE^1\|^2_I+\|\delta_2^h\mH^1\|^2_I\Big\}
    +\frac12\big(\|\delta_x\mE^1\|_E^2+\|\delta_x\mH^1\|^2_H+\|\mE^1\|_I^2+\|\mH^1\|^2_I\big)\\
&&\qquad\qquad+C(\dt)^2\big[(\dt)^4+(\dx)^4 +(\dy)^4+(\dz)^4\big].
\enn
%\normalsize
This together with (\ref{4.13a}) implies the estimate (\ref{4.13}).
The proof is thus complete.
\end{proof}

\begin{remark}\label{rk4.1}{\rm
It is a superconvergence result that the ADI-FDTD scheme is second-order convergent in
space under the discrete $H^1$ semi-norm (see Theorems \ref{t4.1} and \ref{t4.2}).
This is consistent with the superconvergence result of the semi-discrete Yee scheme
established in \cite{Monk1} since both the ADI-FDTD scheme and the Yee scheme are based
on the same spatial discretization technique.
}
\end{remark}

Setting $\delta_w=I$ with $w=x,y,z$ in Theorems \ref{t4.1} and \ref{t4.2} gives
the following optimal error estimates in the discrete $L^2$ norm, which can be proved
by arguing similarly as in the proof of Theorems \ref{t4.1} and \ref{t4.2} with $\delta_w=I$.

\begin{theorem}\label{t4.3}
Suppose the exact solution ${\bf e}$, ${\bf h}$ satisfies the same conditions as in Theorem $\ref{t4.1}.$
Then for any $n\geq 0$ and any fixed $T>0$ there is a constant $C$ independent of $\dt,\dx,\dy,\dz$
such that
\be\no
&&\|\mE^{n+1}\|_E^2+\|\mH^{n+1}\|^2_H
  +\frac{(\dt)^2}{4\mu\vep}(\|\delta_1^h\mE^{n+1}\|_H^2
  +\|\delta_2^h\mH^{n+1}\|^2_E)\\ \label{t4.3a}
&&\qquad\qquad\quad\le C\{(\dt)^4+(\dx)^4+(\dy)^4+(\dz)^4\},\\ \no
&&\|\delta_t\mE^{n+\frac12}\|_E^2+\|\delta_t\mH^{n+\frac12}\|^2_H
  +\frac{(\dt)^2}{4\mu\vep}(\|\delta_1^h\delta_t\mE^{n+\frac12}\|_H^2
  +\|\delta_2^h\delta_t\mH^{n+\frac12}\|^2_E)\\ \label{t4.3b}
&&\qquad\qquad\quad\le C\{(\dt)^4+(\dx)^4+(\dy)^4+(\dz)^4\}.
\en
\end{theorem}

%The following corollary of Theorems \ref{t4.1} and \ref{t4.2} show that the
%perturbation terms in the equivalent form of the ADI-FDTD scheme are bounded.
%
%\begin{corollary}\label{cor4.1}
%Let $\E^n,\Hb^n$ be the solution to the ADI-FDTD scheme and let $\delta_1^h$, $\delta_2^h$
%be the difference operators defined in Section $\ref{sec3.2}.$ Then
%%the second-order and third-order difference quotients of $\E^n$ and $\Hb^n$
%%in the following forms are bounded,
%\be\label{c4.1a}
%\|\delta_1^h\delta_1^h\E^n\|^2_H+\|\delta_2^h\delta_2^h\Hb^{n}\|^2_E&\le& C\\ \label{c4.1b}
%\|\delta_1^h\delta_1^h\frac{\E^{n+1}-\E^n}{\dt}\|^2_H
%+\|\delta_2^h\delta_2^h\frac{\Hb^{n+1}-\Hb^n}{\dt}\|^2_E&\le& C,
%\en
%where $C$ is a constant independent of $\dx,\dy,\dz,\dt$ and $n.$
%\end{corollary}

\section{Divergence preserving property of the ADI-FDTD scheme}\label{div-p}

As a consequence of the Maxwell equations (\ref{mw2.1})-(\ref{mw2.6}),
if the electric field ${\bf e}$ and the magnetic field ${\bf h}$ (multiplied with $\vep$ and $\mu$
respectively) start out divergence free, they will remain so during wave propagation at
any time and any place in the domain $\Om$.
In this section, we will prove that this property is preserved with the second-order accuracy
in both space and time by the ADI-FDTD scheme.

\begin{theorem}\label{t5.1}
Let $\E^n$ and $\Hb^n$ be the solution of the ADI-FDTD scheme and let $\nabla^h\cdot(\vep\E^n)$ and
$\nabla^h\cdot(\mu\Hb^n)$ denote their discrete divergence:
%\small
\ben\label{5.1}
\nabla^h\cdot(\vep\E^n)&=&\vep(\delta_xE^n_x+\delta_yE^n_y+\delta_zE^n_z)\big|_{i,j,k},\\
\nabla^h\cdot(\mu\Hb^n)&=&\mu(\delta_xH^n_x+\delta_yH^n_y
    +\delta_zH^n_z)\big|_{\bar{i},\bar{j},\bar{k}}.
\enn
%\normalsize
Then
%\small
\be\label{5.1a}
\|\nabla^h\cdot(\vep\E^n)\|^2&\le& C\big\{\|\nabla^h\cdot(\vep{\bf e}^0)\|^2
    +(\dt)^4+(\dx)^4+(\dy)^4+(\dz)^4\big\},\\ \label{5.1b}
\|\nabla^h\cdot(\mu\Hb^n)\|^2&\le& C\big\{\|\nabla^h\cdot(\mu{\bf h}^0)\|^2
    +(\dt)^4+(\dx)^4+(\dy)^4+(\dz)^4\big\},
\en
%\normalsize
where for a grid function $U$,
$\ds\|U\|^2=\sum_{i=1}^{I-1}\sum_{j=1}^{J-1}\sum_{k=1}^{K-1}\vep(U_{i,j,k})^2\Delta v.$
%\ben
%\|U\|^2=\sum\limits_{i=1}^{I-1}\sum\limits_{j=1}^{J-1}\sum\limits_{k=1}^{K-1}
%\vep(U_{i,j,k})^2\Delta v.
%\enn
\end{theorem}

\begin{proof}
We only prove (\ref{5.1a}). The inequality (\ref{5.1b}) can be proved similarly.
%Let ${\bf e}=(e_x,e_y,e_z)$ and ${\bf h}=(h_x,h_y,h_z)$ denote the exact solution of
%the Maxwell equations.

From the property that $\nabla\cdot(\vep{\bf e})=\nabla\cdot(\vep{\bf e}^0)$ it follows that
\ben
\|\nabla^h\cdot(\vep{\bf e}^n)\|^2\le\|\nabla^h\cdot(\vep{\bf e}^0)\|^2
   +C\{(\dx)^4+(\dy)^4+(\dz)^4\}.
\enn
Thus, and by Theorem \ref{t4.1} it is derived that
%\small
\ben
\|\nabla^h\cdot(\vep\E^n)\|^2&=&\|\nabla^h\cdot[\vep({\bf e^n}-\E^n+{\bf e^n})]\|^2\\
  &\le&2\big(\|\nabla^h\cdot(\vep\mE^n)\|^2+\|\nabla^h\cdot(\vep{\bf e^n})\|^2\big)\\
  &\le& C\{\|\nabla^h\cdot(\vep{\bf e}^0)\|^2+(\dt)^4+(\dx)^4+(\dy)^4+(\dz)^4\}
\enn
%\normalsize
for some generic constant $C>0$, which is (\ref{5.1a}).
The proof is thus complete.
\end{proof}

%\begin{remark}\label{rk5.1}
%By Theorem \ref{t6.1}, if the initial values of the electric and magnetic fields are divergence
%free, that is, $\nabla^h\cdot(\vep\E^0)_{i,j,k}=\nabla^h\cdot(\vep\Hb^0)_{i,j,k}=0$,
%then $\nabla^h\cdot(\vep\E^n)\big|_{i,j,k}$ and $\nabla^h\cdot(\vep\Hb^n)\big|_{i,j,k})$ are
%also zero in the first-order in time.
%%approximately with the first-order accuracy in time.
%\end{remark}

\section{Numerical experiments}\label{exp}

In this section we carry out numerical experiments to verify the energy identities,
optimal error estimates and divergence preserving properties of the ADI-FDTD scheme.

\subsection{An exact solution and its energy conservation}

Consider the Maxwell equations (\ref{mw2.1})-(\ref{mw2.6}) with $\vep=\mu=1$.
%and the PEC boundary condition (\ref{pec2.1}) covering the domain
%$\Omega=[0,1]\times[0,1]\times[0,1].$
%It is easy to see that the exact solution of the system is given by
The exact solution is given by
%\small
\ben
e_x&=&\frac{1}{4}\sqrt{3}\cos(\sqrt{3}\pi t)\cos[\pi(1-x)]\sin[\pi(1-y)]\sin[\pi(1-z)],\\
e_y&=&\frac{1}{2}\sqrt{3}\cos(\sqrt{3}\pi t)\sin[\pi(1-x)]\cos[\pi(1-y)]\sin[\pi(1-z)],\\
e_z&=&-\frac{3}{4}\sqrt{3}\cos(\sqrt{3}\pi t)\sin[\pi(1-x)]\sin[\pi(1-y)]\cos[\pi(1-z)],\\
h_x&=&-\frac{5}{4}\sin(\sqrt{3}\pi t)\sin[\pi(1-x)]\cos[\pi(1-y)]\cos[\pi(1-z)],\\
h_y&=&\sin(\sqrt{3}\pi t)\cos[\pi(1-x)]\sin[\pi(1-y)]\cos[\pi(1-z)],\\
h_z&=&\frac{1}{4}\sin(\sqrt{3}\pi t)\cos[\pi(1-x)]\cos[\pi(1-y)]\sin[\pi(1-z)].
\enn
%\normalsize

It is easy to verify that the above solution has the following energy conservation
and divergence-free properties:
%\small
\ben
&&\|{\bf e}\|^2=\frac{21}{64}\cos^2(\sqrt{3}\pi t),\quad
    \|{\bf h}\|^2=\frac{21}{64}\sin^2(\sqrt{3}\pi t),\quad
    \nabla\cdot{\bf e}=\nabla\cdot{\bf h}=0,\\
&&E^2_{eh}=|{\bf e}\|^2+\|{\bf h}\|^2=\frac{{21}}{64},\quad
  E^2_{eht}=\|\frac{\pa{\bf e}}{\pa t}\|^2+\|\frac{\pa{\bf h}}{\pa t}\|^2
   =\frac{63}{64}\pi^2,\\
&&E^2_{ehx}=\|\frac{\pa{\bf e}}{\pa x}\|^2+\|\frac{\pa{\bf h}}{\pa x}\|^2
   =\frac{{21}}{64}\pi^2,\quad
E^2_{ehxt}=\|\frac{\pa^2{\bf e}}{\pa x\pa t}\|^2+\|\frac{\pa^2{\bf h}}{\pa x\pa t}\|^2
   =\frac{63}{64}\pi^4,
\enn
%\normalsize
where for a vector function ${\bf u}=(u_x,u_y,u_z)$ the $L^2$-norm $\|\cdot\|$ is defined by
%\small
\ben
\|{\bf{u}}\|^2=\int_\Om P({\bf u})\Big(|u_x|^2+|u_y|^2+|u_z|^2\Big)dx,\;\;
P({\bf e})=\vep,\;P({\bf e})=\mu.
\enn
%\normalsize
$E_{ehw}$ and $E_{ehwt}$ can be defined similarly. Note that $E_{ehw}=E_{ehx}$,
$E_{ehwt}=E_{ehxt}$, $w=y,z.$

\subsection{Energy identities of the ADI-FDTD scheme}

Denote by $\|\cdot\|_1$ and $\|\cdot\|_2$ the two new norms defined in Theorems \ref{t3.1}
and \ref{t3.3} respectively. Then
%\small
\ben
\|(\E^n,\Hb^n)\|_1^2&=&\|\delta_x\E^{n}\|^2_E+\|\delta_x\Hb^{n}\|^2_H
           +\|\E^{n}\|^2_I+\|\Hb^{n}\|^2_I\\
&&+\frac{(\dt)^2}{4\mu\vep}\Big(\|\delta_x\delta_1^h\E^{n}\|^2_H
  +\|\delta_x\delta_2^h\Hb^{n}\|^2_E+\|\delta_2^h\Hb^n\|^2_I+\|\delta_1^h\E^{n}\|^2_I\Big)\\
\|(\E^n,\Hb^n)\|_2^2&=&\|\E^{n}\|^2_E+\|\Hb^{n}\|^2_H
   +\frac{(\dt)^2}{4\mu\vep}\Big(\|\delta_2^h\Hb^{n}\|^2_E+\|\delta_1^h\E^{n}\|^2_H\Big).
\enn
%\normalsize
Similarly, we can define $\|(\delta_t\E^{\bar{n}},\delta_t\Hb^{\bar{n}})\|_1^2$ and
$\|(\delta_t\E^{\bar{n}},\delta_t\Hb^{\bar{n}})\|_2^2$.
%with $\bar{n}=n+1/2$.
To verify the four energy identities established in Subsection \ref{sec3.3},
we introduce the following notations:
% of energy errors:
\small
\ben
&&EH^{n0}_j=\frac{\|(\E^n,\Hb^n)\|_j-\|(\E^0,\Hb^0)\|_j}{\|(\E^0,\Hb^0)\|_j},\;
     EH_{1}^n=\frac{\|(\E^n,\Hb^n)\|_1}{E_{ehx}},\;
     EH_{2}^n=\frac{\|(\E^n,\Hb^n)\|_2}{E_{eh}},\\
&&EH^{n0}_{tj}=\frac{\|(\delta_t\E^{n+\frac12},\delta_t\Hb^{n+\frac12})\|_j
    -\|(\delta_t\E^{\frac12},\Hb^{\frac12})\|_j}
      {\|(\delta_t\E^{\frac12},\delta_t\Hb^{\frac12})\|_j},\;
     EH_{tj}^n=\frac{\|(\delta_t\E^{\bar{n}},\delta_t\Hb^{\bar{n}})\|_j}{E_{eht}},j=1,2.
\enn
\normalsize
Here, $EH^{n0}_j,$ $EH^{n0}_{tj},$ $j=1,2,$ represent the relative error of the two new
energy norms of the solution to the ADI-FDTD scheme at time level $t^n$ and at the initial
time $t^0$ or $t^1$, and $EH^n_j,$ $EH^n_{tj},$ $j=1,2,$ stand for the ratio of the two new
energy norms of the solution of the ADI-FDTD scheme to the corresponding energy norms of
the exact solution.
%which verifies if the discrete energy of the approximate solution is equal to the true energy
%of the exact solution.

The computational results are presented in Table \ref{t1} for different time levels $n$ ($n=T/\dt$),
where the solution to the ADI-FDTD scheme is computed with the spatial step sizes
$\dx=\dy=\dz=0.01$ and the time step size $\dt=0.01$ (note that the Courant number is $\sqrt{3}$).
From Table \ref{t1} it is seen that the four discrete energy identities shown in
Theorems \ref{t3.1}, \ref{t3.2}, \ref{t3.3} and \ref{t3.4} are satisfied
and the discrete energy of the approximate solution is almost equal to the corresponding
ones of the exact solution (their ratio is almost equal to $1$).

\begin{table}[ht]
\begin{center}
\caption{Relative error of the discrete energy of the ADI-FDTD solution at time level $n$ and
at the initial time, and the ratio of the discrete energy of the ADI-FDTD solution to the
true energy of the exact solution under the norms $\|\cdot\|_1$ and $\|\cdot\|_2$ at
different time levels with $\dx=\dy=\dz=\dt=0.01$}\label{t1}
\renewcommand{\arraystretch}{1.2}
\begin{tabular}{|c|ccccr|}
\hline
Energy$\backslash$Time & n=100 & n=400& n=800&n=1600&n=2000\\
\hline
$EH_1^{n0}$& 3.157e-13 & 4.780e-14 & 2.783e-13& 2.537e-13&1.646e-13\\
$EH_1^n$&1.00008 & 1.00008 & 1.00008 &1.00008&1.00008\\
\hline
$EH_{t1}^{n0}$ & 1.627e-12 & 1.518e-12 & 1.428e-12& 1.478e-12&1.559e-12\\
$EH_{t1}^n$ &0.99979 & 0.99979 & 0.99979 &0.99979&0.99979\\
\hline
$EH_2^{n0}$ & 1.397e-13 & 2.808e-13 & 3.432e-13& 2.254e-13&2.167e-13\\
$EH_2^n$ &1.0001 &1.0001 & 1.0001 &1.0001&1.0001\\
\hline
$EH_{t2}^{n0}$ & 1.457e-12 & 1.585e-12 & 1.744e-12& 1.438e-12&1.634e-12\\
$EH_{t2}^n$ &0.99984 & 0.99984 & 0.99984 &0.99984&0.99984\\
\hline
\end{tabular}
\end{center}
\end{table}

To verify the energy identities without the perturbation terms, define
%denote them by $\|\|_{1*}$ and $\|\cdot\|_{2*}$ defined as
%\small
\ben
\|(\E^n,\Hb^n)\|_{1*}^2&:=&\|\delta_x\E^{n}\|^2_E+\|\delta_x\Hb^{n}\|^2_H
           +\|\E^{n}\|^2_I+\|\Hb^{n}\|^2_I,\\
\|(\E^n,\Hb^n)\|_{2*}^2&:=&\|\E^{n}\|^2_E+\|\Hb^{n}\|^2_H.
\enn
%\normalsize
Similarly, let $EH^{n0}_{j*},$ $EH^{n0}_{tj*},$ $j=1,2,$ represent the relative error of the
two new energy norms without the perturbation terms of the solution to the ADI-FDTD scheme
at time level $t^n$ and at the initial time $t^0$ or $t^1$ and let
$EH^n_{j*},$ $EH^n_{tj*},$ $j=1,2,$ stand for the ratio of the two new energy norms
without the perturbation terms of the solution of the ADI-FDTD scheme to the exact ones.

Table \ref{tb} shows similar results to those in Table \ref{t1} with the four discrete
energy identities without the perturbation terms.
These new discrete energy norms can be seen as the discrete forms of the physical energies.
From these values we can see that the energy identities without the perturbation terms are
still satisfied and that the discrete energies of the ADI-FDTD solution are almost equal
to the corresponding ones of the exact solution.

\begin{table}[ht]
\begin{center}
\caption{Relative error of the discrete  energy of the ADI-FDTD solution at time level $n$ and
at the initial time, and the ratio of the discrete energy of the ADI-FDTD solution to the
physical energy of the exact solution under the norms $\|\cdot\|_{1*}$ and $\|\cdot\|_{2*}$ at
different time levels with $\dx=\dy=\dz=\dt=0.01$}\label{tb}
\renewcommand{\arraystretch}{1.2}
\begin{tabular}{|c|ccccr|}
\hline
Energy$\backslash$Time & n=100 & n=400& n=800&n=1600&n=2000\\
\hline
$EH_{1*}^{n0}$& 1.592e-14 & 3.140e-13 & 1.730e-14& 1.413e-13&1.070e-14\\
$EH_{1*}^n$&0.99996 & 0.99996 & 0.99996 & 0.99996& 0.99996\\
\hline
$EH_{t1*}^{n0}$ & 1.942e-13 & 3.359e-13 & 2.508e-13& 2.899e-13&4.480e-13\\
$EH_{t1*}^n$ &0.99967 & 0.99967 & 0.99967 &0.99967&0.99967\\
\hline
$EH_{2*}^{n0}$ & 1.463e-13 & 1.421e-13 & 1.531e-13& 2.105e-13& 6.590e-15\\
$EH_{2*}^n$ &1.0000 &1.0000 & 1.0000 &1.0000&1.0000\\
\hline
$EH_{t2*}^{n0}$ & 9.150e-14 & 9.050e-14 & 2.309e-13& 2.269e-13&3.705e-15\\
$EH_{t2*}^n$ &0.99971 & 0.99984 & 0.99971 &0.99971& 0.99971\\
\hline
\end{tabular}
\end{center}
\end{table}

%
%\begin{table}[ht]
%\begin{center}
%\caption{Relative error of energy of in $H^1$ and $L^2$ at
%different time period}\label{t.5}
%\begin{tabular}{|l|ccccr|}
%\hline time domain & EnH1-I & EnH1-II& EnL2-I&EnL2-II &R-Energy\\
%\hline
%T=1& -3.157e-13 & 1.397e-13 & 6.272e-13& 5.821e-13&1.00\\
%     T=4   &-4.799e-14 & 2.808e-13 & 3.184e-13 &5.520e-13&1.0e+3\\
%       T=10 & -2.518e-13&1.696e-13 & 1.561e-12& 1.436e-12&1.0e+3\\
%       T=20 & -1.646e-13 &2.166e-13 &1.559e-12& 1.634e-12&1.0e+3\\
%\hline
%\end{tabular}
%\end{center}
%\end{table}

\subsection{Error and convergence rate of the ADI-FDTD scheme}
%under different norms}

To verify the convergence rate of the ADI-FDTD scheme under different norms,
let
\ben
&&\ds\mE\mH^{n}_1={\|(\mE^n,\mH^n)\|_1}/{E_{ehx}},\qquad
\mE\mH^{n}_{t1}={\|(\delta_t\mE^{n+\frac12},\delta_t\mH^{n+\frac12})\|_1}/{E_{ehxt}},\\
&&\ds\mE\mH^{n}_2={\|(\mE^n,\mH^n)\|_2}/{E_{eh}},\qquad
\ds\mE\mH^{n}_{t2}={\|(\delta_t\mE^{n+\frac12},\delta_t\mH^{n+\frac12})\|_2}/{E_{eht}},\\
&&\ds\mE\mH^{n}_0=(\|\mE^n\|^2_E+\|\mH^n\|^2_H)/E_{eh}.
\enn
%\ben
%&&\mE\mH^{n}_1=\frac{\|(\mE^n,\mH^n)\|_1}{E_{ehx}},\qquad
% \mE\mH^{n}_{t1}=\frac{\|(\delta_t\mE^{n+\frac12},\delta_t\mH^{n+\frac12})\|_1}{E_{ehxt}},\\
%&&\mE\mH^{n}_2=\frac{\|(\mE^n,\mH^n)\|_2}{E_{eh}},\qquad
% \mE\mH^{n}_{t2}=\frac{\|(\delta_t\mE^{n+\frac12},\delta_t\mH^{n+\frac12})\|_2}{E_{eht}},\\
%&&\mE\mH^{n}_0=\frac{\|(\mE^n,\mH^n)\|_0}{E_{eh}},\qquad
% \|(\mE^n,\mH^n)\|_0^2=\|\mE^n\|^2_E+\|\mH^n\|^2_H.
%\enn

Table \ref{t2} gives the relative error in different norms of the ADI-FDTD solution
at different time levels with the time step size $\dt=0.01$, the spatial step sizes
$\dx=\dy=\dz=0.01$ and the Courant number $\sqrt{3}$.
The results demonstrate that even for a large Courant number and at a long time, e.g. $T=20$,
the ADI-FDTD solution is still stable and accurate in different norms.

\begin{table}[ht]
\begin{center}
\caption{Relative error in the norms $\|\cdot\|_1,$ and $\|\cdot\|_2$
 of the ADI-FDTD solution at different time with $\dx=\dy=\dz=0.01$ and $\dt=0.01$}\label{t2}
\renewcommand{\arraystretch}{1.2}
\begin{tabular}{|c|ccccr|}
\hline
Error$\backslash$Time & n=100 & n=400& n=800&n=1600&n=2000\\
\hline
$\me\mh^n_1$ & 9.091e-4 & 3.580e-3 & 7.161e-3& 1.432e-2&1.790e-2\\
$\me\mh^n_{t1}$ &9.162e-4 & 3.585e-3 & 7.159e-3 &1.431e-2&1.789e-2\\
\hline
$\me\mh^n_2$ & 3.214e-4 & 1.266e-3 & 2.532e-3& 5.063e-3&6.329e-3\\
$\me\mh^n_{t2}$ &9.159e-4 & 3.585e-3 & 7.195e-3 &1.431e-2&1.789e-2\\
%$\me\mh_0^n$& 9.090e-4 & 3.58e-3 & 7.160e-3&1.432e-2&1.790e-2\\
\hline
\end{tabular}
\end{center}
\end{table}

The convergence rates in time and space of the ADI-FDTD scheme are
given in Tables \ref{t3}-\ref{t8} under different norms.
The results show that the convergence rate in the norms $\|\cdot\|_1$,
$\|\cdot\|_2$ and $\|\cdot\|_0$ is second order in space and time, which
confirm the theoretical results. Note that the second-order convergence in time
and space under the norm $\|\cdot\|_1$ is a superconvergence result.

\begin{table}[ht]
\begin{center}
\caption{Convergence rate in time under the norm $\|\cdot\|_1$ at $T=1$
with $\dx=\dy=\dz=0.01$}\label{t3}
\begin{tabular}{|c|cc||cr|}
\hline
$\dt$& $\mE\mH^n_1$ & Rate& $\mE\mH^n_{t1}$& Rate\\
\hline
0.05& 1.749e-2 &  & 1.706e-2& \\
0.04& 1.127e-2 & 1.971 & 1.108e-2 & 1.934\\
0.025& 2.978e-3 & 1.938 & 2.975e-3&1.913 \\
0.02& 2.978e-3 & 1.881 & 2.975e-3&1.863 \\
\hline
\end{tabular}
\end{center}

%\end{table}
\vspace{0.15cm}
%\begin{table}[ht]
\begin{center}
\caption{Convergence rate in space under the norm $\|\cdot\|_1$ at $T=1$
with $\dt=0.001$}\label{t4}
\begin{tabular}{|c|cc||cr|}
\hline $\dx=\dy=\dz$& $\mE\mH^N_1$ & Rate& $\me\mh^N_{t1}$& Rate\\
\hline
0.025& 1.405e-3 &  & 1.430e-3& \\
0.02& 9.016e-4 & 1.987 & 9.221e-4 & 1.966\\
0.01& 2.305e-4 & 1.968 & 3.141e-4&1.554 \\
\hline
\end{tabular}
\end{center}
\end{table}

%\vspace{0.15cm}
\begin{table}[ht]
\begin{center}
\caption{Convergence rate in time under the norm $\|\cdot\|_{2}$ at $T=1$ with
$\dx=\dy=\dz=0.01$}\label{t5}
\begin{tabular}{|c|cc||cc|}
\hline $\dt$& $\me\mh^n_{2}$ & Rate& $\me\mh^n_{t2}$&  Rate\\
\hline
0.050& 6.185e-3 &  & 1.706e-2& \\
0.040& 3.984e-3 &  1.971& 1.108e-2 & 1.934\\
0.025& 2.978e-3 & 1.938 & 2.975e-3&1.913 \\
0.020& 1.053e-3 & 1.881 & 2.975e-3&1.863 \\
\hline
\end{tabular}
\end{center}
%\end{table}

\vspace{0.15cm}
%\begin{table}[ht]
\begin{center}
\caption{Convergence rate in space under the norm $\|\cdot\|_{2}$ at $T=1$ with
$\dt=0.001$}\label{t6}
\begin{tabular}{|c|cc||cc|}
\hline $\dx=\dy=\dz$& $\mE\mH^N_{2}$ &  Rate& $\mE\mH^N_{t2}$& Rate\\
\hline
0.025& 4.968e-4 &  & 1.428e-3& \\
0.020& 3.188e-4 &  1.988& 9.165e-4 & 1.987\\
0.010& 8.149e-5 & 1.968 & 2.342e-4&1.968 \\
\hline
\end{tabular}
\end{center}
\end{table}

\begin{table}[ht]
\begin{center}
\caption{Convergence rate in time under the norm $\|\cdot\|_{0}$ at $T=1$ with
$\dx=\dy=\dz=0.01$}\label{t7}
\begin{tabular}{|c|cc||cc|}
\hline $\dt$ & $\|\mE^n\|_E$ & Rate& $\|\mH^n\|_E$ & Rate\\
\hline
0.050& 1.914e-2 &  & 1.595e-2& \\
0.040& 1.232e-2 &  1.975& 1.031e-2 & 1.956\\
0.025& 4.955e-3 & 1.938 & 4.157e-3&1.933 \\
0.020& 3.260e-3 & 1.876 & 2.730e-3&1.883 \\
\hline
\end{tabular}
\end{center}
%\end{table}
%\begin{table}[ht]

\vspace{0.15cm}
\begin{center}
\caption{Convergence rate in space under the norm $\|\cdot\|_{0}$ at $T=1$ with
$\dt=0.001$}\label{t8}
\begin{tabular}{|c|cc||cc|}
\hline $\dx=\dy=\dz$ & $\|\mE^N\|_E$ & Rate& $\|\mH^N\|_E$ & Rate\\
\hline
0.025& 4.968e-4 &  & 1.7064e-2& \\
0.020& 3.188e-4 &  1.988& 1.1082e-2 & 1.934\\
0.010& 8.149e-5 & 1.968 & 2.975e-3&1.913 \\
\hline
\end{tabular}
\end{center}
\end{table}

\subsection{Divergence of the numerical electric field}

To verify the discrete divergence-free property of the ADI-FDTD scheme,
we introduce two norms for divergence:
%\small
\ben
\mbox{Div}_{L^\infty}&=&\max\limits_{i,j,k}\{\vep|\delta_xE^M_x
      +\delta_yE^M_y+\delta_zE^M_z|_{i,j,k}\},\\
\mbox{Div}_{L^2}&=&\Big(\sum\limits_{i=1}^{I-1}\sum\limits_{j=1}^{J-1}\sum\limits_{k=1}^{K-1}
   \vep|\delta_xE^M_x+\delta_yE_y^M+\delta_zE_z^M|^2_{i,j,k}\Delta v\Big)^{1/2}.
\enn
%\normalsize
Table \ref{t9} gives the values of the above norms at different times $T=1,4,8$, $16$ and $20$,
where $M\dt=T$. The results show that the discrete divergence is almost zero, which means
that the ADI-FDTD scheme is of the approximate divergence preserving property.

\begin{table}[ht]
\begin{center}
\caption{Divergence error of the numerical electric field $\E$ in $L^\infty$
and $L^2$ norms at different times}\label{t9}
\renewcommand{\arraystretch}{1.2}
\begin{tabular}{|c|ccccr|}
\hline Error/Time & T=1 & T=4& T=8&T=16&T=20\\
\hline
$\mbox{Div}_{L^\infty}$ & 7.060e-13 & 1.421e-12 & 1.943e-12& 2.956e-12&2.912e-12\\
$\mbox{Div}_{L^2}$ & 7.421e-14 & 1.540e-13 & 2.176e-13& 3.085e-13&3.451e-13\\
\hline
\end{tabular}
\end{center}
\end{table}

\section*{Acknowledgements}

The work of the first author (LG) was supported by NSF of Shandong Province grant Y2008A19
and Research Reward for Excellent Young Scientists from Shandong Province (2007BS01020).
This work started when LG was with School of Mathematical Sciences at Shandong Normal
University, China. The work of the second author (BZ) was supported by
NNSF of China grant 11071244.

\end{document}